\documentclass[aos,preprint]{imsart}

\RequirePackage{amsthm,amsmath,amsfonts,amssymb}
\RequirePackage[round,authoryear]{natbib}
\RequirePackage[colorlinks,citecolor=blue,urlcolor=blue]{hyperref}
\RequirePackage{graphicx}
\usepackage{bbm,bm}
\usepackage{enumerate}
\usepackage{caption,subcaption}

\startlocaldefs
\theoremstyle{plain}

\newtheorem{thm}{Theorem}[section]
\newtheorem{lemma}[thm]{Lemma}
\newtheorem{remark}{Remark}[section]

\newtheorem{cor}[thm]{Corollary}

\theoremstyle{remark}

\endlocaldefs

\newcommand{\bbV}{{\bf V}}
\newcommand{\bbU}{{\bf U}}
\newcommand{\bbD}{{\bf D}}
\newcommand{\bbA}{{\bf A}}

\newcommand{\bbX}{{\bf X}}
\newcommand{\bbP}{{\bf P}}

\newcommand{\bbZ}{{\bf Z}}
\newcommand{\bbx}{{\bf x}}

\newcommand{\bbS}{{\bf S}}
\newcommand{\bbB}{{\bf B}}

\newcommand{\bSi}{\pmb \Sigma}

\newcommand{\bbalp}{{\bf \alpha}}

\newcommand{\bgl}{{\bf \lambda}}

\newcommand{\bGma}{{\bf \Gamma}}
\newcommand{\bPsi}{{\bf \Psi}}

\newcommand{\bgO}{{\bf \Omega}}

\newcommand{\bUps}{{\bf \Upsilon}}

\newcommand{\bbI}{{\bf I}}

\newcommand{\bbT}{{\bf T}}

\newcommand{\bqn}{\begin{eqnarray*}}
	\newcommand{\eqn}{\end{eqnarray*}}

\newcommand{\rtr}{{\textrm{tr}}}
\newcommand{\rdiag}{{\textrm{diag}}}

\newcommand{\bqa}{\begin{eqnarray}}
	\newcommand{\eqa}{\end{eqnarray}}

\newcommand{\al}{\alpha}

\newcommand{\CYRS}{\CYRS}

\begin{document}

\begin{frontmatter}
\title{A CLT for the LSS of large dimensional sample covariance matrices with unbounded dispersions}
\runtitle{CLT for LSS with unbounded dispersions}

\begin{aug}
\author{\fnms{Zhijun} \snm{Liu}\ead[label=e1,mark]
{liuzj037@nenu.edu.cn}},
\author{\fnms{Jiang} \snm{Hu}\ead[label=e2,mark]{huj156@nenu.edu.cn}},
\author{\fnms{Zhidong} \snm{Bai}\ead[label=e3,mark]{baizd@nenu.edu.cn}},
\and 
\author{\fnms{Haiyan}
\snm{Song}\ead[label=e4,mark]{songhy716@nenu.edu.cn}}
\address{KLASMOE and School of Mathematics and Statistics, Northeast Normal University, China.
\printead{e1,e2,e3,e4}}

\end{aug}

\begin{abstract}
In this paper, we establish the central limit theorem (CLT) for linear spectral statistics (LSS) of large-dimensional sample covariance matrix when the population covariance matrices are not uniformly bounded, which is a nontrivial extension of the Bai-Silverstein theorem (BST) (2004). The latter has strongly stimulated the development of high-dimensional statistics, especially the application of random matrix theory to statistics. However, the assumption of uniform boundedness of the population covariance matrices is found strongly limited to the applications of BST. The aim of this paper is to remove the blockages to the applications of BST. The new CLT, allows the spiked eigenvalues to exist and tend to inﬁnity. It is interesting to note that the roles of either spiked eigenvalues or the bulk eigenvalues or both of the two are dominating in the CLT.

Moreover, the results are checked by simulation studies with various population settings. The CLT for LSS is then applied for testing the hypothesis that a covariance matrix $ \bSi $ is equal to an identity matrix.  For this, the asymptotic distributions for the corrected likelihood ratio test (LRT)
and Nagao's trace test (NT) under alternative are derived, and we also propose the asymptotic power of LRT and NT under certain alternatives. 



\end{abstract}

\begin{keyword}[class=MSC]
\kwd[Primary ]{	60B20}
\kwd[; secondary ]{60F05}
\end{keyword}

\begin{keyword}
\kwd{Empirical spectral distribution}
\kwd{linear spectral statistic}
\kwd{random matrix}
\kwd{Stieltjes transform}
\end{keyword}
\end{frontmatter}

\section{Introduction}

Consider the general sample covariance matrix  $ \bbB_n=\frac{1}{n}\bbT_p\bbX_n\bbX_n^{\ast}\bbT_p^{\ast} $, where $ \bbX_n $ is a $ p\times n $ matrix with independent and identically distributed (i.i.d.) entries, $ \bbT_p $ is a $p \times p$ deterministic matrix, $\bbT_p\bbX_n$ can be considered a random sample from the population with the population covariance matrix $\bbT_p\bbT_p^{\ast}=\bSi$, and $^*$ represents
the complex conjugate transpose. In this sequel, we will simply write $\bbB\equiv\bbB_n$, $\bbT\equiv\bbT_p$ and $\bbX\equiv\bbX_n$ when there is no confusion. $ \lambda_{1},\lambda_{2},\ldots,\lambda_{p} $ denotes the eigenvalues of $\bbB$. For a known kernel function $f$, we call $\sum_{j=1}^{p}f\left( \bgl_{j}\right) $ the linear spectral statistic (LSS) of $\bbB$.  As most of the classical test statistics in multivariate statistical analysis are associated with the eigenvalues of sample covariance matrices,  LSSs are remarkable tools in many statistical problems (see  \cite{Anderson03I,YaoB15L} for details). By extensively studying high-dimensional data, 
it was found that the performances of the LSSs are significantly different between low dimensions and high-dimensional data.
For example, under the low-dimensional setting, Wilks' theorem (see \cite{Wilks38L}) provides the $\chi^2$ approximation for the likelihood ratio statistics (LRT), which is a kind of LSS. However,  when $ p $ is large compared with the sample size $ n $, the LRTs have Gaussian fluctuations (see \cite{10.1214/09-AOS694,JiangY13C}). 
More generally, \cite{10.1214/aop/1078415845} proved the central limit theorem for the LSSs of high-dimensional $ \bbB $ under very mild conditions using the random matrix theory (RMT). We refer to this result as the Bai-Silverstein theorem (BST) for brevity. Following the development of \cite{10.1214/aop/1078415845}, there are many extensions under different settings.  \cite{pan2008central} generalized the BST by removing the constraint on the fourth moment of the underlying random variables.  \cite{zheng2012central} and \cite{YangP15I} extended the BST to multivariate $F$ matrices and canonical correlation matrices, respectively. \cite{pan2014comparison} showed the CLT of the LSS for non-centered sample covariance matrices, and \cite{ 10.1214/14-AOS1292} studied the unbiased sample covariance matrix when the population mean is unknown. \cite{chen2015clt} focused on the ultra-high dimensional case in which the dimension $p$ is much larger than the sample size $n$. \cite{GaoH17H} and \cite{LiW21C} studied the CLT for the LSSs of the high-dimensional Spearman correlation and Kendall's rank correlation matrices, respectively. Without attempting to be comprehensive, we refer readers to other extensions \citep{BaiM07A, BaiH15C, BaiL19C, ZhengC19T, BannaN20C, najim2016gaussian, baik2018ferromagnetic, hu2019high, JiangB21G}.

Almost all the theories mentioned above have traditionally assumed that the spectral norms of $ \bSi $ are bounded in $ n $.
This assumption severely limits their applications in data analysis because in many fields, such as economics and wireless communication networks, the leading eigenvalues may tend to infinity. We use three examples here.
\begin{itemize}
	\item  \textit{\textbf{Panel data model}}(\cite{doi:10.1080/07474938.2017.1307580}): Consider a fixed effect panel data model $$ y_{i t}=x_{i t}^{\prime} \beta+\mu_{i}+v_{i t},~~i=1, \ldots, n,~~t=1, \ldots, T, $$ where $i$ is the index of the cross-sectional units, $t$ is the index of the time series observations, $\mu_{i}$ represents the time invariant individual effects, and $v_{i t}$ is the idiosyncratic error term. ${\bf v}_{t}=(v_{1 t},\dots,v_{n t})^{\prime} $ are i.i.d. $ N(0,\bSi). $  Assume that $ v_{i t}=\sum_{j=1}^{r}\gamma_{ij}f_{tj}+\epsilon_{t} $, where $ f_{tj} $ is the factor $ j $ in period $ t $, $ \gamma_{ij} $ is the factor loading of the individual $ i $ for factor $ j $, $ \epsilon_{t} $ is the error term with i.i.d. $N(0,\sigma^{2}$), and $ r $ is the known number of factors. 
	The sphericity test in the fixed effect panel data model is
	 $$ H_{0}: \bSi=\sigma^{2}\bbI_{n}~~\mbox{v.s.}~~H_{1}:  \bSi=\sigma^{2}\left(\bbI_{n}+\sum_{j=1}^{r}\frac{\sigma_j^2}{\sigma^2}\bm\gamma_{j}\bm\gamma_{j}'\right),$$
	  where $ \bbI_{n} $ denotes the $ n $-dimensional identity matrix. ${\bm\gamma}_j=(\gamma_{1j},\dots,\gamma_{nj})'$ is the vector of factor loading, and $\sigma_j^2$ is the variance of factor $f_{ tj}$. Many efforts have been made to analyze the asymptotic power of sphericity tests in high-dimensional setups, where the number of cross-sectional units $ n $ in a panel is large, but the number of time series observations $ T $ could also be large. When $ n $ jointly tends to infinity with $ T $,  the norm of the perturbation term in the alternative hypothesis is greater than the threshold or even goes to infinity. In this case, the existing methods that assume $ \bSi $ are bounded are not applicable.
	 
\item \textit{\textbf{Signal detection}}(\cite{10.1093/biomet/asw060}): Consider a single signal model
	\begin{equation*} 
		{\bf x}=\chi_{s}^{1/2}u{\bf h} +\sigma\bm{v},   
	\end{equation*} 
	where $ {\bf h} $ is an unknown $ p $-dimensional vector, $ u $ is a random variable distributed as $ N(0,1) $, $  \chi_{s}$ is the signal strength, $\sigma  $ is the noise level, and $\bm {v} $ is
 a random noise vector that is independent of $ u $ and distributed as a multivariate Gaussian $N_p (0, \bm\Sigma_{v})$.
 It is easy to check that the covariance matrix of $\bbx$ is $\bm\Sigma_x=\sigma^2\bm\Sigma_{v}+\chi_s{\bf h}{\bf h}'$. When the noise level is low, 
while the signal strength is large and sometimes tends to infinity, it is illogical to assume the boundedness of $\bm\Sigma_x$.

\item \textit{\textbf{$ \textbf{m} $-factor structure}}(\cite{li2020asymptotic}): Consider the $ m $-factor model
 $$
 \boldsymbol{X}_{t}=\mathbf{A} \boldsymbol{F}_{t}+\boldsymbol{E}_{t}
 $$
 where the factors $\boldsymbol{F}_{t} \sim N\left(0, \mathbf{I}_{m}\right)$  are independent of the idiosyncratic error terms $\boldsymbol{E}_{t} \sim$ $N\left(0, \sigma^{2} \mathbf{I}_{p}\right) .$ The loading matrix $\mathbf{A}_{p \times m}$ is deterministic and of full rank such that $\mathbf{A}^* \mathbf{A}$ has eigenvalues $a_{1}>\cdots>a_{m}>0$. The eigenvalues of the population covariance matrix $\bSi_{p}$ of $\boldsymbol{X}_{t}$ are
 $$\operatorname{Spec}\left(\bSi_{p}\right)=\{a_{1}+\sigma^{2}, \dots, a_{m}+\sigma^{2}, \underbrace{\sigma^{2}, \dots, \sigma^{2}}_{p-m}\},$$ 
  which follow the generalized spiked model. Because of the complexity of the real data, when $\sigma$ is small, the signal-to-noise ratio may be large enough to give rise to the large spectral norm of $\bSi_{p}/\sigma^2$. Therefore, in this case, we assume the unbounded spectra of the population covariance matrices would be more realistic.
\end{itemize}
For these reasons, it is of practical value to obtain the asymptotic properties of the LSS when $ \bSi $ is unbounded. Therefore, in this paper,
we focus on the CLT for the LSS of a general covariance matrix structure
\begin{align}\label{ds}
  \bSi=\mathbf{V}\left(\begin{array}{cc}
	\bbD_{1} & 0 \\
	0 & \bbD_{2}
\end{array}\right) \mathbf{V}^{\ast},
\end{align}
where $\mathbf{V}$ is a unitary matrix, $\bbD_{1}$ is a diagonal matrix consisting of the descending unbounded eigenvalues, and $\bbD_{2}$ is the diagonal matrix of the bounded eigenvalues. 
 The setting \eqref{ds} is attributed to the famous spiked model
in which a few large eigenvalues of the population covariance matrix are assumed to be well separated from the rest \citep{10.1214/aos/1009210544}. The spiked model has provided the foundations for a rich theory of principal component analysis through the performance of extreme eigenvalues as discussed in \cite{BAIK20061382, 10.2307/24307692, 10.1214/07-AIHP118, Nadler08F,JungM09P,  BAI2012167, OnatskiM14S, BloemendalK16P, WangY17E, DonohoG18O, PerryW18O, JohnstoneP18P, YangJ18E, YaoZ18E, Dobriban20P, JohnstoneO20T, cai2020limiting, JiangB21G}. Recently, \cite{li2020asymptotic}, \cite{yin2021spectral} and \cite{zhang2022asymptotic} investigated the trace of the large sample covariance matrix with the spiked model assumption.

{ We highlight the main contributions of the present paper. First, we prove a non-trivial extension of the BST to the case in which the population covariance matrices are unbounded in the spectral norm. In particular, we show how the kernel functions and the divergence rate of the population spectral norm affect the new CLT. 
Second, it is known that Gaussian-like moments, i.e., the first fourth moments, coincide with a standard Gaussian distribution, or the diagonality of the population covariance matrix is necessary for the CLT of the LSSs (e.g., \cite{10.1214/14-AOS1292}). However, we prove that these restrictions can be completely removed by renormalization. More importantly, even if the limit of the LSS variance does not exist, the renormalized CLT still holds. Third, 
 by combining the technical strategy in
\cite{10.1214/aop/1078415845} and the analysis of the block decomposition of the sample covariance matrix in \cite{JiangB21G}, we prove that the LSSs of the unbounded and bounded parts are asymptotically independent.} 
 The proof in this entire paper is built on the decomposition of the LSS, which is divided into two parts, an unbounded part and a bounded part. It is worth noting that the bounded part of the LSS cannot use the result in \cite{10.1214/aop/1078415845} directly since off-diagonal sample covariance matrix blocks are not 0, which yields a bias between the bounded part of the LSS and the LSS of the bounded sample covariance matrix blocks. In facing this challenge, we make full use of the RMT and prove for the first time that bias can be measured in probability in the literature. 
As an application, the established CLT is employed to study the asymptotic behavior of the LRT and Nagao's trace (NT) test under the hypothesis $$ H_{0}:\bSi=\bbI_{p}~~\mbox{v.s.}~~H_{1}:\bSi\neq\bbI_{p}.$$ It is known that the LRT and NT are typical examples of LSSs with kernel functions $ f(x)=x-\log x-1 $ and $ x^{2}-2x+1 $.  In this paper, we start from a different perspective by studying the asymptotic distribution of the LRT and NT under the alternative that $ \bbD_{1} $ is composed of $ M $ spiked eigenvalues tending to infinity with multiplicity, $ \bbD_{2}=\bbI_{p-M} $, and then we establish their asymptotic power under the above alternative. 
Based on previous knowledge (e.g., \citep{10.1214/09-AOS694}), it seems that when the number of spiked eigenvalues is small, the influence caused by the spiked part is small, and the distribution of the LSS is mainly decided by the bulk part; however, that is not the case. After simulation, a surprising result is that when the spiked eigenvalues are very large, the LSS will also be affected by the spiked part even though the number of spikes is small. 

The remaining sections are organized as follows: Section \ref{section 2} provides a detailed description of the notation and assumptions. The main results of the CLTs for the LSS of the sample covariance matrix are stated in Section \ref{section 3}. In Section \ref{section 5}, we explore the applications of our main results. We also present the results of our numerical studies in Section \ref{section 4}. Technical proofs of the theorems are presented in Section \ref{section 6}.

\section{Notation and assumptions}\label{section 2}

 Throughout the paper, we use bold capital letters and bold lowercase letters to represent matrices and vectors, respectively. Scalars are often in regular letters.  $\boldsymbol{e}_{i}$ denotes a standard basis vector whose components are all zero, except the $i$-th,
which equals 1. We use tr$ (\bbA) $, $ \bbA' $ and $ \bbA^{\ast} $ to denote the trace, transpose and conjugate transpose of matrix $ \bbA $, respectively. We also use $f'$ to denote the derivative of function $f$, and we use $ \frac{\partial}{\partial z_{1}}f(z_{1},z_{2}) $ to denote the partial derivative of function $ f $ with respect to $ z_{1} $; however, the context is clear enough that there is no risk of ambiguity. Let $ \left[\bbA \right]_{ij}  $ denote the $ (i,j) $-th entry of the matrix $ \bbA $ and $ \oint_{\mathcal{C}}f(z)dz $ denote the contour integral of $ f(z) $ on the contour $ \mathcal{C} $. Let $ \lambda_{i}^{\bbA} $ denote the $ i $th largest eigenvalue of matrix $ \bbA $. Weak convergence is denoted by $ \stackrel{d}{\rightarrow}$. Throughout this paper,  we use $o(1)$ (resp. $o_p (1)$) to denote a scalar negligible (resp. in probability), and the notation $C$ represents some generic constants that may vary from line to line.

Let $ \bbX=(\boldsymbol x_{1},\ldots,\boldsymbol x_{n})=(x_{ij}) $, $ 1\leq i\leq p $, $ 1\leq j\leq n $, and $ \bbT $ be a $p \times p$ deterministic matrix and $\bSi=\bbT\bbT^{\ast}$. The spectrum of $ \bSi $ is formed as $\rho_{ 1}\geq\cdots \geq \rho_{ p}$. 
Define the singular value decomposition of $\bbT$ as
\begin{equation}\label{decT}
	\bbT=\bbV\bbD^{1/2}\bbU^*=
	\bbV
	\left( 
	\begin{array}{cc}
		\bbD_{1}^\frac{1}{2} & 0\\
		0 & \bbD_{2}^\frac{1}{2}
	\end{array}
	\right)
	\bbU^{\ast}
\end{equation}
where  $\bbU$ and $\bbV$ are unitary matrices, $\bbD_{1}=\rdiag(\alpha_1,\dots,\alpha_1,\alpha_2,\dots,\alpha_2,\dots,\alpha_K,\dots,\alpha_K)$ is a diagonal matrix of the spiked eigenvalues for which the components tend to infinity, and $\bbD_{2}$ is the diagonal matrix of the eigenvalues with the bounded components. Here,
$\al_{1}>\cdots>\al_{K}$ denotes the unbounded spiked eigenvalues of $ \bSi $ with the multiplicity $ d_{k}, k=1,\ldots,K $, and $ d_{1}+\cdots+d_{K}=M$. 
Moreover, let $ \rho_{i} = \al_{k} $ if $ i \in J_{k} $, where $ J_{k}=\left\lbrace j_{k}+1, \ldots, j_{k}+d_{k}\right\rbrace $ is the set of ranks of the $d_{k}$-ple eigenvalue $\alpha_{k}$. Then, 
 the corresponding sample covariance matrix $$\bbB=\frac{1}{n}\bbT\bbX\bbX^{\ast}\bbT^{\ast}$$
is the so-called generalized spiked sample covariance matrix.  $\lambda_{1}\geq\lambda_{2}\geq\cdots \geq\lambda_{p}$ denotes the eigenvalues of $\bbB$. 
{ Corresponding to the decomposition of $\bbD$, we decompose $ \bbV=\left(\bbV_{1},\bbV_{2} \right)$,  $ \bbU=\left(\bbU_{1},\bbU_{2} \right)$, and $\bGma=\bbV_{2}\bbD_{2}^{1/2}\bbU_{2}^{\ast} $, $ \boldsymbol r_{j}=\frac{1}{\sqrt{n}}\bGma\boldsymbol x_{j} $, $ \bbA_{j}=\bbB-z\bbI-\boldsymbol r_{j}\boldsymbol r_{j}^{*}. $ Let $ \mathbb{E}_{j} $ be the conditional expectation with respect to the $ \sigma $-field generated by $ \boldsymbol r_{1}, \dots, \boldsymbol r_{j} $. 
%
For any matrix $\bbA $ with real eigenvalues, the empirical spectral distribution of $\bbA $ is denoted by
\begin{equation*}
	F^{\bbA}\left(x \right)=\frac{1}{p}\left(\text{number of eigenvalues of }  \bbA \leq x \right).  	
\end{equation*}	
For any function of bounded variation $F$ on the real line, its Stieltjes transform is defined by
$$
m_{F}(z)=\int \frac{1}{\lambda-z} \mathrm{~d} F(\lambda), \quad z \in \mathbb{C}^{+} :=\{z \in \mathbb{C}: \Im z>0\}.
$$

The assumptions used in the results of this paper are as follows:
\newtheorem{assumption}{Assumption}[]
\begin{assumption} \label{ass1}
	$ \{x_{ij},  1\leq i\leq p ,  1\leq j\leq n \} $ {are independent random variables with common moments} $$ \mathbb{E}x_{ij}=0,\quad  \mathbb{E}\left| x_{ij}\right| ^{2}=1, \quad  \beta_{x}= \mathbb{E}\left|x_{ij}\right| ^{4}- \left|\mathbb{E}x_{ij}^{2}\right|^{2}-2,  \quad \alpha_{x}=\left|\mathbb{E}x_{ij}^{2}\right|^{2},  $$ 
{ and satisfy the following Lindeberg-type condition: }
	\begin{align*}
		&\frac{1}{n p} \sum_{i=1}^{p} \sum_{j=1}^{n} \mathbb{E}\left\{\left|x_{ij}\right|^{4} \mathbbm{1}_{\left\{\left|x_{ij}\right| \geq \eta \sqrt{n}\right\}}\right\} \rightarrow 0, \quad \text { for any constant } \eta>0.
	\end{align*}	
\end{assumption} 
\begin{assumption} \label{ass2}
	{ As $\min\{p,n\}\to\infty$, the ratio of the dimension-to-sample size (RDS)} $ c_{n}:={p}/{n}\rightarrow c>0. $
\end{assumption} 
\begin{remark}
	Assumptions 1-2 are standard in the RMT. Note that if $\mathbb{E}x_{ij}\neq 0$, we can use the centralized sample covariance matrices and $n-1
	$ instead of $\bbB_n$ and $n$, respectively, and the following results also hold. Details can be found in \citep{10.1214/14-AOS1292}. Therefore, in this sequel, we assume $\mathbb{E}x_{ij}= 0$ without loss of generality. 
\end{remark}
\begin{assumption}\label{ass3}
	 {$ \bbT $ is non-random. As $\min\{p,n\}\to\infty$, $\alpha_K\to\infty$ and $H_n:=F^{\bGma\bGma^{*}}\stackrel{d}{\rightarrow}H$, where $H$ is a distribution function on the real line.  $M$ is fixed.}
%
\end{assumption} 

\begin{remark}
It was shown by \cite{Silverstein95S} that under Assumptions 1-3, $F^{\bbB}\stackrel{d}{\rightarrow}F^{c,H}$ almost surely, where $F^{c,H}$ is a non-random distribution function whose Stieltjes transform $m:=m_{F^{c,H}}(z)  $ satisfies equation
\begin{align}\label{Sequation}
  m=\int \frac{1}{t(1-c-c z m)-z} d H(t).
\end{align}
In this sequel, we call $F^{c,H}$ the limiting spectral distribution (LSD) of $\bbB$.
Moreover, as the matrix $\underline{\bbB}=\frac{1}{n}\bbX^*\bbT^{*}\bbT\bbX$ shares the same non-zero eigenvalues with $\bbB$, equation \eqref{Sequation} can be rewritten as
$$\underline{m}=-\left(z-c \int \frac{t }{1+t \underline{m}}d H(t)\right)^{-1},$$ 
where $\underline{m}:=m_{\underline{F}^{c,H}}(z)$ represents the Stieltjes transform of the LSD of $\underline{\bbB}$.
\end{remark} 
\begin{assumption}\label{ass4}
	Kernel functions $ f_{1},\dots, f_{h} $ are analytic on an open domain of the complex plan containing the support of $ F^{c,H}. $ Moreover, suppose that for any $l=1,\dots,h$,  $$\lim_{\{x_{n},y_n\}\to\infty\atop {x_{n}}/{y_{n}}\rightarrow 1}\frac{f_{l}'\left(x_{n} \right) }{f_{l}'\left(y_{n} \right)}= 1 .$$
\end{assumption}    
\begin{remark}
	In fact, Assumption \ref{ass4} is not too restrictive for application, many common functions such as logarithmic and polynomial functions satisfy it. However, it is worth noting that the exponential function may not satisfy this assumption.
\end{remark}

\begin{assumption}\label{ass5}
	$ \bbT $  {is real or the variables} $x_{ij}$ {are complex satisfying} $\alpha_{x}=0 .$ 
\end{assumption}
\begin{assumption}\label{ass6}
	$ \bbT^{*}\bbT $  {is diagonal or} $ \beta_{x}=0. $
\end{assumption}
\begin{remark}
	\textit{Assumption \ref{ass5} is for the second-order moment condition of $ x_{ij} $, and Assumption \ref{ass6} is for the fourth-order moment. They were first proposed by} \cite{10.1214/14-AOS1292},  \textit{ who proved that the two assumptions are necessary for their results when the Gaussian-like moment conditions in the BST do not hold. }
\end{remark}

\section{Main Results}  \label{section 3}
Now, we are in a position to present our main theorems, and their proofs are provided in Section \ref{section 6}. Note that $$\sum_{j=1}^{p}f\left( \bgl_{j}\right)=p\int f\left(x \right)dF^{\bbB}(x). $$ 
Thus, for brevity, we define the normalized LSSs as
$$  Y_{l}= \int f_{l}\left(x \right)dG_{n}\left( x\right)-{\sum_{k=1}^{K}d_{k}f_{l}\left(\phi_{n}\left(\al_{k} \right)  \right)}-\frac{M}{2\pi i}\oint_{\mathcal C}f_{l}\left(z \right)\frac{\underline{m}'(z)}{\underline{m}(z)}dz.\quad l=1,2,\dots,h, $$
where $$ G_{n}\left( x\right)=p[F^{\bbB}\left(x \right)-F^{c_{n},H_{n}}\left(x \right)]   , ~~\phi_n\left(x \right)=x\left(1+c_n\int\frac{t}{x-t}dH_n\left(t \right)  \right), $$ 
and $F^{c_{n},H_{n}}$ is the LSD $F^{c,H}$ with $c,~H$ replaced by $c_n,~ H_n$. $ \underline{F}^{c,H} $ denotes the LSD of matrix $ n^{-1}\bbX^{\ast}\bbU_{2}\bbD_{2}\bbU_{2}^{*}\bbX$, $ \bbU=(\bbU_{1},\bbU_{2}), $
$ \bbU_{1}=\left(u_{ij} \right)_{i=1,\dots,p;j=1,\dots,M}  $, $ \phi_{k}=\phi\left(x \right)\mid_{x=\al_{k}}=\al_{k}\left(1+c\int\frac{t}{\al_{k}-t}dH\left(t \right)  \right)  $,
$$ b_{n}\left(z \right)=\frac{1}{1+n^{-1}\mathbb{E}\rtr\bGma\bGma^{*}\bbA_{j}^{-1}\left(z \right) },~~~{\mathcal{U}_{i_{1}j_{1}i_{2}j_{2}}=\sum_{t=1}^{p}\overline{u}_{ti_{1}}u_{tj_{1}}u_{ti_{2}}\overline{u}_{tj_{2}},  } $$ $$\theta_{k}=\phi_{k}^{2}\underline{m}_{2}\left( \phi_{k}\right), ~~\nu_{k}=\phi_{k}^{2} \underline{m}^{2}\left(\phi_{k}\right),$$$$  \underline{m}\left( \lambda\right)=\int\frac{1}{x- \lambda }d\underline{F}^{c,H}\left( x\right),~~\underline{m}_{2}\left( \lambda\right)=\int\frac{1}{\left( \lambda-x\right) ^{2}}d\underline{F}^{c,H}\left( x\right), $$ 
$$  c_{nM}=\dfrac{p-M}{n},~~H_{2n}=F^{\bbD_{2}},  $$
$$ \bbP_{n}(z)=((1-c_{nM})\bGma\bGma^{*}-zc_{nM}m_{2n0}(z)\bGma\bGma^{*}-z\bbI_{p})^{-1}.  $$
Here, $ m_{2n0}(z) $ is the Stieltjes transform of $ F^{c_{nM},H_{2n}} $ and $ \underline{m}_{2n0}(z)=-\frac{1-c_{nM}}{z}+c_{nM}m_{2n0}(z) $. For clarification purposes, $ m_{1n0}(z) $ also denotes the Stieltjes transform of $ F^{c_{n},H_{n}} $, $m_{n}=\frac{1}{p}\mathrm{tr}\left(\bbB-z\bbI_{p} \right)^{-1} $, and $m_{2n}=\frac{1}{p-M}\mathrm{tr}\left(\bbS_{22}-z\bbI_{p-M} \right)^{-1}$, which will be used later in the proof. 

We first establish a CLT of the LSS
without any restrictions imposed on the Gaussian-like moments 
or on the structures of the population covariance matrix  by renormalizing the LSS.
 

\begin{thm}\label{thm1}
	\textit{Under Assumptions
	 \ref{ass1}--\ref{ass4}}, we have $$ \frac{Y_{1}-\mu_{1}}{\sqrt{\varsigma_{1}^{2}}}\stackrel{d}{\rightarrow}N\left(0, 1\right),    $$  \textit{where the mean function is}
	\begin{align}
		\mu_{1}&=\nonumber-\frac{\alpha_{x}}{2 \pi i}\cdot\oint_{\mathcal{C}}\frac{  c_{nM}f_{1}(z) \int \underline{m}_{2n0}^{3}(z)t^{2}\left(1+t \underline{m}_{2n0}(z)\right)^{-3} d H_{2n}(t)}{\left(1-c_{nM} \int \frac{\underline{m}_{2n0}^{2}(z) t^{2}}{\left(1+t \underline{m}_{2n0}(z)\right)^{2}} d H_{2n}(t)\right)\left(1-\alpha_{x} c_{nM} \int \frac{\underline{m}_{2n0}^{2}(z) t^{2}}{\left(1+t \underline{m}_{2n0}(z)\right)^{2}} d H_{2n}(t)\right) }dz \\
		&-\frac{\beta_{x}}{2 \pi i} \cdot \oint_{\mathcal{C}} \frac{c_{nM}  f_{1}(z)\int \underline{m}_{2n0}^{3}(z) t^{2}\left(1+t \underline{m}_{2n0}(z)\right)^{-3} d H_{2n}(t)}{1-c_{nM} \int \underline{m}^{2}_{2n0}(z) t^{2}\left(1+t \underline{m}_{2n0}(z)\right)^{-2} d H_{2n}(t)} dz, \label{33} \quad 
	\end{align}
	\textit{and the covariance function is} 
	\begin{align*}
		\varsigma_{1}^{2}=\sum_{k=1}^{K}\frac{\phi_{n}^{2}\left(\al_{k} \right)}{n}\left( f_{1}'\left(\phi_{n}\left(\al_{k} \right)\right)  \right)^{2}s_{k}^{2} -\frac{1}{4\pi^{2}}\oint_{\mathcal{C}_{1}}\oint_{\mathcal{C}_{2}}f_{1}\left(z_{1} \right)f_{1}\left(z_{2} \right)\vartheta_{n}^{2}dz_{1}dz_{2}.
	\end{align*} 
	Here, $ s_{k}^{2}=\frac{\sum_{j\in J_{k}}\left( \left(\al_{x}+1 \right)\theta_{k}+\beta_{x}\mathcal{U}_{jjjj}\nu_{k}\right) +\sum_{j_{1}\neq j_{2}}\beta_{x}\mathcal{U}_{j_{1}j_{1}j_{2}j_{2}}\nu_{k} }{\theta_{k}^{2}}, \vartheta_{n}^{2}=\Theta_{0,n}(z_{1},z_{2})+\al_{x}\Theta_{1,n}(z_{1},z_{2})+\beta_{x}\Theta_{2,n}(z_{1},z_{2}),$ where
	\begin{align*}
		\Theta_{0,n}(z_{1},z_{2})&=\dfrac{\underline{m}_{2n0}^{\prime}(z_{1}) \underline{m}_{2n0}^{\prime}(z_{2})  }{(\underline{m}_{2n0}(z_{1})-\underline{m}_{2n0}(z_{2}) )^{2} }-\dfrac{1}{(z_{1}-z_{2})^{2}},\\
		\Theta_{1,n}(z_{1},z_{2})&=\frac{\partial}{\partial z_{2}}\left\lbrace \dfrac{\partial \mathcal{A}_{n}(z_{1},z_{2})}{\partial z_{1}}\dfrac{1}{1-\al_{x}\mathcal{A}_{n}(z_{1},z_{2})} \right\rbrace ,\\
		{\mathcal{A}_{n}(z_{1},z_{2})}&=\dfrac{z_{1}z_{2}}{n}\underline{m}_{2n0}^{\prime}(z_{1}) \underline{m}_{2n0}^{\prime}(z_{2})\mathrm{tr}{\bGma^{*}\bbP_{n}(z_{1})\bGma\bGma^{\prime}\bbP_{n}(z_{2})^{\prime} \bar{\bGma}}\\
		\Theta_{2,n}(z_{1},z_{2})&=\dfrac{z_{1}^{2}z_{2}^{2}\underline{m}_{2n0}^{\prime}(z_{1}) \underline{m}_{2n0}^{\prime}(z_{2})}{n}\sum_{i=1}^{p}\left[ \bGma^{*}\bbP_{n}^{2}(z_{1})\bGma\right]  _{ii}\left[ \bGma^{*}\bbP_{n}^{2}(z_{2})\bGma\right]  _{ii}
	\end{align*}
	In the equation above, $\mathcal{C}, \mathcal{C}_{1}$ \textit{and} $\mathcal{C}_{2}$ \textit{are closed contours in the complex plan enclosing the support of the} $ F^{c, H}$, and $\mathcal{C}_{1}$ \textit{and }$\mathcal{C}_{2}$ \textit{are nonoverlapping}. It is worth noting that $ \Theta_{1,n}(z_{1},z_{2})  $ and $ \Theta_{2,n}(z_{1},z_{2}) $ may not converge.
\end{thm}

As a minor price for the applicability enlargement, the new CLT described above only applies to a single LSS. To guarantee that the new CLT applies to multiple normalized LSSs,  structural
assumptions about the population covariance matrices are needed (Assumptions \ref{ass5} and \ref{ass6}). 
The following theorem is a non-trivial extension of the BST:

\begin{thm}\label{thm2}
	\textit{Under Assumptions \ref{ass1}--\ref{ass6}}, the random vector  $$ \left( \frac{Y_{1}-\mu_{1}}{\sqrt{\sigma_{1}^{2}}},\dots,\frac{Y_{h}-\mu_{h}}{\sqrt{\sigma_{h}^{2}}}\right)'\stackrel{d}{\rightarrow}N_{h}\left(0, \bPsi \right),    $$
    with mean function   \begin{align}
    	\mu_{l}&=\nonumber-\frac{\alpha_{x}}{2 \pi i}\cdot\oint_{\mathcal{C}}f_{l}(z)\frac{  c_{nM} \int \underline{m}_{2n0}^{3}(z)t^{2}\left(1+t \underline{m}_{2n0}(z)\right)^{-3} d H_{2n}(t)}{\left(1-c_{nM} \int \frac{\underline{m}_{2n0}^{2}(z) t^{2}}{\left(1+t \underline{m}_{2n0}(z)\right)^{2}} d H_{2n}(t)\right)\left(1-\alpha_{x} c_{nM} \int \frac{\underline{m}_{2n0}^{2}(z) t^{2}}{\left(1+t \underline{m}_{2n0}(z)\right)^{2}} d H_{2n}(t)\right) }dz \\
    	&-\frac{\beta_{x}}{2 \pi i} \cdot \oint_{\mathcal{C}} f_{l}(z) \frac{c_{nM} \int \underline{m}_{2n0}^{3}(z) t^{2}\left(1+t \underline{m}_{2n0}(z)\right)^{-3} d H_{2n}(t)}{1-c_{nM} \int \underline{m}^{2}_{2n0}(z) t^{2}\left(1+t \underline{m}_{2n0}(z)\right)^{-2} d H_{2n}(t)} dz, \label{33} \quad l=1,\dots,h,	
    \end{align}
    variance function
	\begin{align*}
		\sigma_{l}^{2}&=\sum_{k=1}^{K}\frac{\phi_{n}^{2}(\al_{k})}{n}(f_{l}^{\prime }(\phi_{n}(\al_{k})))^{2} s_{k}^{2}+\kappa_{ll}, \quad l=1,\dots,h,
    \end{align*} 
     and covariance matrix $  \bPsi =\left( \psi_{st}\right) _{h\times h} $, 
     	\begin{align*}
		\psi_{st}=\frac{\sum_{k=1}^{K}\varpi_{s}^{k}\varpi_{t}^{k}s_{k}^{2}+ \kappa_{st}}{\sqrt{\sum_{k=1}^{K}(\varpi_{s}^{k})^{2}s_{k}^{2} +\kappa_{ss}}\sqrt{\sum_{k=1}^{K}(\varpi_{t}^{k})^{2}s_{k}^{2} +\kappa_{tt}}},
	\end{align*} 
	\textit{where}	
	\begin{align*}
		\kappa_{st}
		&=-\frac{1}{4 \pi^{2}} \oint_{\mathcal{C}_{1}} \oint_{\mathcal{C}_{2}} \frac{f_{s}\left(z_{1}\right) f_{t}\left(z_{2}\right)}{\left(\underline{m} {\left.\left(z_{1}\right)-\underline{m}\left(z_{2}\right)\right)^{2}}\right.} d \underline{m}\left(z_{1}\right) d \underline{m}\left(z_{2}\right) 
		-\frac{c \beta_{x}}{4 \pi^{2}} \oint_{\mathcal{C}_{1}} \oint_{\mathcal{C}_{2}} f_{s}\left(z_{1}\right) f_{t}\left(z_{2}\right)\\
		&\times\left[\int \frac{t^2}{\left(\underline{m}\left(z_{1}\right) t+1\right)^{2}\left(\underline{m}\left(z_{2}\right) t+1\right)^{2}} d H(t)\right] d \underline{m}\left(z_{1}\right) d \underline{m}\left(z_{2}\right)\\
		&\quad-\frac{1}{4 \pi^{2}} \oint_{\mathcal{C}_{1}} \oint_{\mathcal{C}_{2}} f_{s}\left(z_{1}\right) f_{t}\left(z_{2}\right)\left[\frac{\partial^{2}}{\partial z_{1} \partial z_{2}} \log \left(1-a\left(z_{1}, z_{2}\right)\right)\right] d z_{1} d z_{2},\\
		a\left(z_{1}, z_{2}\right)&=\alpha_{x}\left(1+\frac{\underline{m}\left(z_{1}\right) \underline{m}\left(z_{2}\right)\left(z_{1}-z_{2}\right)}{\underline{m}\left(z_{2}\right)-\underline{m}\left(z_{1}\right)}\right).
		\end{align*} 
Note that $ s_{k}^{2} $ is defined in Theorem \ref{thm1}, and $ \varpi_{l}^{k}=\lim_{n\to\infty}\dfrac{\phi_{n}(\al_{k})}{\sqrt{n}}f_{l}^{\prime}(\phi_{n}(\al_{k}) )$ is allowed to be infinity.
	
\end{thm}

\begin{remark}\label{remark}

If we rewrite the covariance matrix $  \bPsi$ in Theorem \ref{thm2} in the form of expressions with $n$, such as $\varsigma_{1}^{2}$ in Theorem \ref{thm1}, then Assumptions 5-6 should be removed analogously. However, in this case, we need to verify the nonsingularity of these covariance matrices, which is difficult unless the kernel functions are linearly related. Thus, we decided not to pursue that direction in this paper.

\end{remark}
\begin{remark}
	\textit{It is not difficult to find from the theorem that the asymptotic distributions of the LSSs depend on the divergence rates of} $ { \phi_{n}\left(\al_{k} \right)}$ and $ f'_{l}\left(\phi_{n}\left(\al_{k} \right) \right)  $. In particular, when $ \frac{ \phi_{n}\left(\al_{1} \right)}{\sqrt{n}} f'_{l}\left(\phi_{n}\left(\al_{1} \right) \right)\rightarrow 0 \left( n\rightarrow \infty \right)$, Theorem \ref{thm2} reduces to Theorem 2.1 in \cite{10.1214/14-AOS1292}; when $ \frac{ \phi_{n}\left(\al_{1} \right)}{\sqrt{n}} f'_{l}\left(\phi_{n}\left(\al_{1} \right) \right)\rightarrow \infty \left( n\rightarrow \infty \right)$, Theorem \ref{thm2} is a non-trivial extension of the CLT derived by \cite{JiangB21G}. Furthermore,  \cite{yin2021spectral} recently obtained a CLT when the kernel functions were polynomial. 
\end{remark}	
\begin{remark}\label{simple result}
	If $ \bbD_{2}=\bbI_{p-M} $, then the mean function $\mu_1$ and $ \kappa_{st} $ in the covariance function of Theorem \ref{thm2}  can be simplified from the results in \cite{wang2013sphericity} and \cite{10.1214/14-AOS1292}, i.e., 
	$$\phi_n\left(x \right)=x+\frac{x(p-M)}{n(x-1)},
~~~~
	\mu_{1}=\alpha_x I_{1}(f_{1})+\beta_{x} I_{2}(f_{1}),
$$
$$\kappa_{st}=(\alpha_x+1)J_{1}(f_{s}, f_{t})+\beta_{x} J_{2}(f_{s}, f_{t}),$$
	\begin{align*}
		I_{1}(f_{1})&=\lim _{r \downarrow 1} \frac{1}{2 \pi i} \oint_{|z|=1} f_{1}\left(\left|1+\sqrt{c} z\right|^{2}\right)\left[\frac{z}{z^{2}-r^{-2}}-\frac{1}{z}\right] d z,\\
		I_{2}(f_{1})&=\frac{1}{2 \pi i} \oint_{|z|=1} f_{1}\left(\left|1+\sqrt{c} z\right|^{2}\right) \frac{1}{z^{3}} d z,\\
		J_{1}(f_{s}, f_{t})&=\lim _{r \downarrow 1} \frac{-1}{4 \pi^{2}} \oint_{\left|z_{1}\right|=1} \oint_{\left|z_{2}\right|=1} \frac{f_{s}\left(\left|1+\sqrt{c} z_{1}\right|^{2}\right) f_{t}\left(\left|1+\sqrt{c} z_{2}\right|^{2}\right)}{\left(z_{1}-r z_{2}\right)^{2}} d z_{1} d z_{2},\\
		J_{2}(f_{s}, f_{t})&=-\frac{1}{4 \pi^{2}} \oint_{\left|z_{1}\right|=1} \frac{f_{s}\left(\left|1+\sqrt{c} z_{1}\right|^{2}\right)}{z_{1}^{2}} d z_{1} \oint_{\left|z_{2}\right|=1} \frac{f_{t}\left(\left|1+\sqrt{c} z_{2}\right|^{2}\right)}{z_{2}^{2}} d z_{2}.
	\end{align*}

\end{remark}

\section{Application}\label{section 5}
In this section, we focus on testing the hypothesis that a high-dimensional covariance matrix $ \bSi $ is equal to an identity matrix, that is,
$$ H_{0}:\bSi=\bbI_{p}\quad \text{vs} \quad H_{1}:\bSi=\mathbf{V}\left(\begin{array}{cc}
	\bbD_{1} & 0 \\
	0 & \bbI_{p-M}
\end{array}\right) \mathbf{V}^{\ast}
,$$ 
where $\bbD_{1}=\rdiag(\alpha_1,\dots,\alpha_1,\alpha_2,\dots,\alpha_2,\dots,\alpha_K,\dots,\alpha_K)$. 
For this problem, the two most classical test statistics are the likelihood ratio  test (LRT) statistic \citep{Wilks38L} and 
Nagao's trace (NT)
 test statistic \citep{Nagao1973}. Specifically, the LR and NT statistics can be formalized as
$$
L=\operatorname{tr} \bbB-\log \left|\bbB\right|-p~~\mbox{and}~~W=\mathrm{tr}(\bbB-\bbI_{p})^{2}, 
$$
respectively.  Under the null hypothesis, we refer to \cite{10.1214/09-AOS694,JiangY13C,Ledoit02,wang2013sphericity,OnatskiM13A} for the asymptotic properties of the LR and NT statistics for high-dimensional settings. In this section, we mainly focus on the alternative hypothesis $H_1$. However,  to provide a better comparison, we also present the asymptotic distributions under the null hypothesis in the following theorems:

\begin{thm}[CLT for LR]\label{thm3}
	Under Assumptions \ref{ass1}-\ref{ass4} with $c_{n}=p / n \rightarrow c \in(0,1)$, we have 
\begin{itemize}
	\item (Under $H_0$,) $$\dfrac{L-p\ell_L -\mu_{L}}{\sqrt{\varsigma_{L}^{2}}} \stackrel{d}{\longrightarrow} N(0,1),  $$
	where 
\begin{align*}
	\ell_L=1-\frac{c_{n}-1}{c_{n}} \log \left(1-c_{n}\right),~~~
	\mu_{L} = -\frac{\log \left(1-c_{n}\right)}{2}\al_{x}+\frac{c_{n}}{2}\beta_{x}
	\end{align*}
	and$$
	\varsigma_{L}^{2} =\frac{\al_{x}+1}{2}(-2 \log \left(1-c_{n}\right)-2 c_{n}).$$
    \item  (Under $H_1$,) $$\dfrac{L-(p-M)\breve{\ell}_L  -\breve{\mu}_{L}}{\sqrt{\breve{\varsigma}_{L}^{2}}} \stackrel{d}{\longrightarrow} N(0,1),  $$
	where 
	\begin{align*}
	\breve{\ell}_L=&1-\frac{c_{nM}-1}{c_{nM}} \log \left(1-c_{nM}\right)\\
\breve{\mu}_{L}=& -\frac{\log \left(1-c_{nM}\right)}{2}\al_{x}+\frac{c_{nM}}{2} \beta_{x} \\
&+\sum_{k=1}^{K}d_{k}\left( \phi_{n}\left({\al}_{k} \right)-\log\phi_{n}\left({\al}_{k} \right)-1 \right)-M(c_{n}+\log(1-c_{n})) \\
	\breve{\varsigma}_{L}^{2} =&\frac{\al_{x}+1}{2}(-2 \log \left(1-c_{nM}\right)-2 c_{nM})+\sum_{k=1}^{K}d_{k}\frac{2\left( \phi_{n}\left({\al}_{k} \right)-1\right)^{4}}{n\phi_{n}^{2}\left({\al}_{k} \right) }\\
	\phi_n\left({\al}_{k}  \right)=&{\al}_{k} +\frac{{\al}_{k} (p-M)}{n({\al}_{k} -1)}.
	\end{align*}	
\end{itemize}
%
\end{thm}

\begin{remark}
	If $ c>1$, then $\bbB_n$ is singular for large $n$, which gives rise to the undefined LR statistic $L$. Thus, in Theorem \ref{thm3}, we add an additional restriction $c<1$. 
\end{remark}
%
%
%
%
%
%
\begin{thm}[CLT for NT]\label{thm4}
	Under Assumptions \ref{ass1}-\ref{ass4}, we have
\begin{itemize}
	\item (Under $H_0$,) $$ \frac{W-p\ell_{W}-\mu_{W}}{ \sqrt{\varsigma_{W}^{2}} }\stackrel{d}{\longrightarrow} N(0,1),   $$
where 
	\begin{align*}
		\ell_{W}=c_{n},~~
		\mu_{W} =\al_{x}c_{n}+\beta_{x}c_{n}~~\mbox{and}~~		\varsigma_{W}^{2} =(\al_{x}+1)(4c_{n}^{3}+2c_{n}^{2})+4\beta_{x}c_{n}^{3}.
	\end{align*}

	\item  (Under $H_1$,)
	\begin{align*}
		\frac{W-(p-M)\breve{\ell}_{W} -\breve{\mu}_{W}}{\sqrt{\breve{\varsigma}_{W}^{2}}}\stackrel{d}{\longrightarrow} N(0,1), 
    \end{align*}	
where  \begin{align*}
	\breve{\ell}_{W}&=c_{nM},\\
	\breve{\mu}_{W}&=\al_{x}c_{nM}+\beta_{x}c_{nM}+\sum_{k=1}^{K}d_{k}\left( \phi_{n}^{2}\left({\al}_{k} \right)-2\phi_{n}\left({\al}_{k} \right)+1 \right)-M c_{n}^{2}\\
	\breve{\varsigma}_{W}^{2}&=(\al_{x}+1)(4c_{nM}^{3}+2c_{nM}^{2})+4\beta_{x}c_{nM}^{3}+\sum_{k=1}^{K}d_{k}\frac{8\left( \phi_{n}\left({\al}_{k} \right)-1\right)^{4}}{n}\\
	\phi_n\left({\al}_{k}  \right)&={\al}_{k} +\frac{{\al}_{k} (p-M)}{n({\al}_{k} -1)}.
	\end{align*}
	
\end{itemize}
%
%
\end{thm}
\begin{remark}\label{remarkonespike}
For illustration, 
we study the asymptotic properties of the LR and NT statistics under a simplified spiked population model \citep{10.1214/aos/1009210544}, where the population covariance matrix is diagonal, with only one spiked eigenvalue, i.e.,
\begin{align} \label{alter2}
	\bSi=\operatorname{diag}(\al_{1},1,\dots,1)  
\end{align}
Let $ z_{a} $ be the upper $ a\% $ quantile of the standard Gaussian distribution $ \Phi$ and $\phi_{1}=\al_{1}+\frac{c\al_{1}}{\al_{1}-1} $. It is straightforward from Theorem \ref{thm3} that the asymptotic power of the LR statistic under the alternative hypothesis \eqref{alter2} is


 $$  \Phi( \frac{ -c+  \left( \phi_{1}-\log\phi_{1}-1 \right)-z_{a}\sqrt{\frac{\al_{x}+1}{2} (-2\log(1-c)-2c) }  }{ \sqrt{\frac{\al_{x}+1}{2} (-2\log(1-c)-2c)  +\frac{2}{n}\frac{(\phi_{1}-1)^{4}}{\phi_{1}^{2}}}   }   )  .$$
Analogously, the asymptotic power of the NT statistic under the alternative hypothesis \eqref{alter2} can be formulized as 
 $$   \Phi(\dfrac{-2c+ (\phi\left(\al_{1} \right)-1)^{2} -c^{2}-z_{a}\sqrt{(\al_{x}+1)(4c^{3}+2c^{2})+4\beta_{x}c^{3} }  }{ \sqrt{(\al_{x}+1)(4c^{3}+2c^{2} )+4\beta_{x}c^{3}+\frac{8}{n}(\phi(\al_{1})-1)^{4}  }  }  ). $$ 
It is easy to find that if $\alpha_1\to\infty$, the asymptotic power of the LR and NT statistics tend to 1 of order $\min\{\alpha_1,\sqrt{n}\}$ and $\min\{\alpha_1^2,\sqrt{n}\}$, respectively. 

\end{remark}

\section{Simulation}\label{section 4}
In this section, we conducted a number of simulation studies  to examine the asymptotic distributions of statistics $L $ and $W$ under $ H_{0} $ and $ H_{1} $ in Section \ref{section 5}. For brevity, under $ H_{1} $, we focus on a simplified spiked population with only one spiked eigenvalue $ \al_{1} $. In our experiments, we define $ \bbZ = \bbT\bbX $, where $ \bbX=(x_{ij})_{p\times n}  $, and $x_{ij}$ are i.i.d. standard normality distributions.
All the simulations are based on $ c=\frac{1}{3} $ and 1,000 repetitions.

The theories in Section \ref{section 5} reveal that for the LR statistic $L= \mathrm{tr}\bbB-\log\left|\bbB \right| -p $, when the spiked eigenvalue is weak, there is only a constant shift in the mean term, but the covariance term remains unchanged. When the spiked eigenvalue divergences of order $ n^{\frac{1}{2}} $, the mean term has a divergence shift, and the covariance term has a constant shift. When $ {\al_{1}}/{n^{\frac{1}{2}}}  $ tends to infinity, the asymptotic distribution is dominated by the spiked eigenvalue. For these properties, we set three cases in the simulation, which are :
\begin{align*}
\mbox{Case 1:}~\alpha_1=3,~~  \mbox{Case 2:}~\alpha_1=n^{1/2},~~ \mbox{Case 3:}~\alpha_1=n^{2/3}. 
\end{align*}
Similar to the discussion of the LR statistic, we know that the distinguishing divergence rate of $\alpha_1$ in the NT statistic $W=\mathrm{tr}(\bbB-\bbI_{p})^{2}$ is $n^{1/4}$. Thus, to examine the properties of $W$,  we set three cases:
\begin{align*}
\mbox{Case 4:}~\alpha_1=3,~~  \mbox{Case 5:}~\alpha_1=n^{1/4},~~ \mbox{Case 6:}~\alpha_1=n^{1/3}. 
\end{align*}
All the results are presented in Figures 1-6.
Note that the kernel density and normal density under each setting are compared using the same color.

We highlight two observations from Figures 1-6. First, all the curves fit the limiting distribution well when the dimension is sufficiently large, which is consistent with Theorem \ref{thm1} and Theorem \ref{thm2}. Second, 
   by comparing Figures 1-3 with Figures 4-6, we find that the NT statistic is more powerful than the LR statistic for the spiked alternative hypothesis, which indicates that different kernel function choices are also important for the testing problem.

\begin{figure}
	\centering
	\begin{minipage}[t]{0.3\textwidth}
		\centering
		\includegraphics[width=4.5cm]{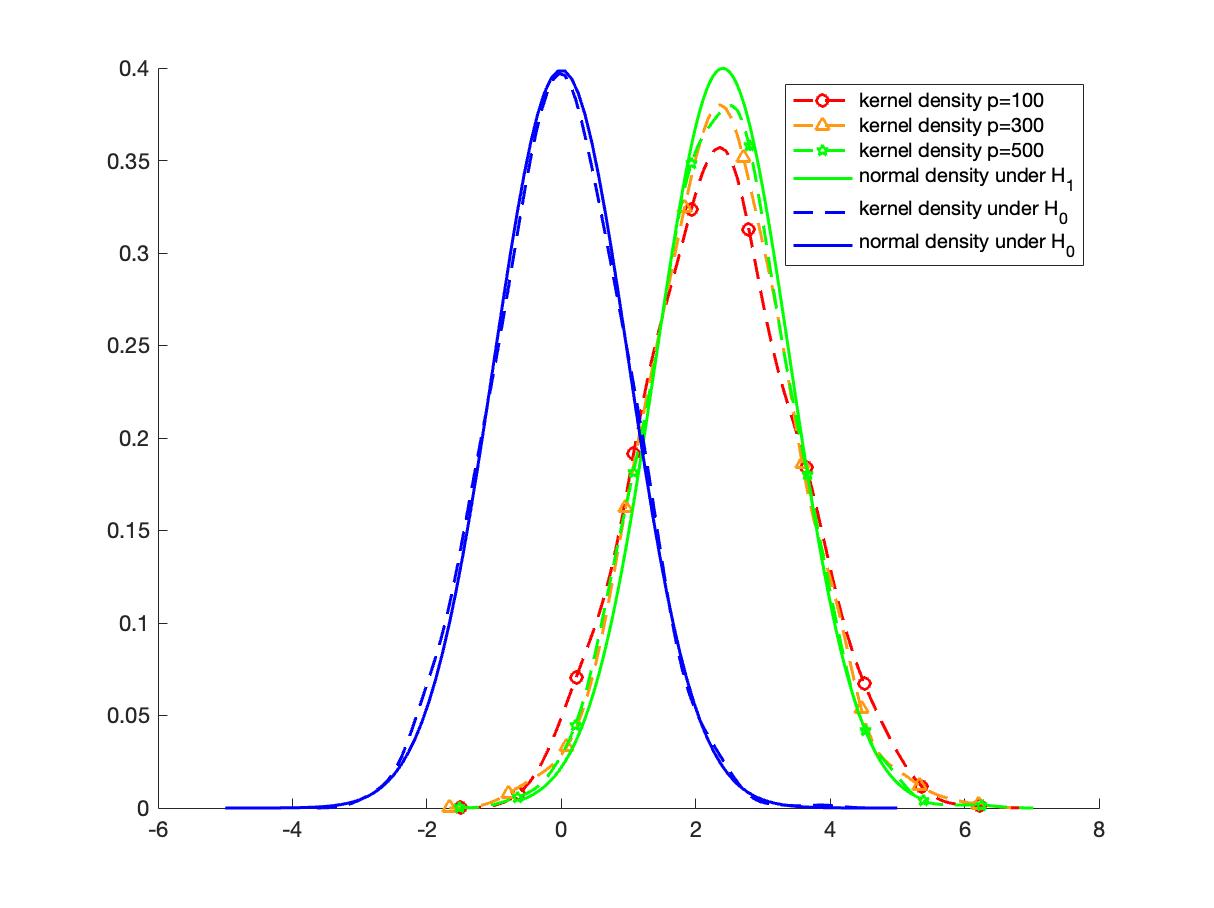}
		\captionsetup{font={scriptsize}}
		\caption{Case 1}
	\end{minipage}
	\begin{minipage}[t]{0.3\textwidth}
		\centering
		\includegraphics[width=4.5cm]{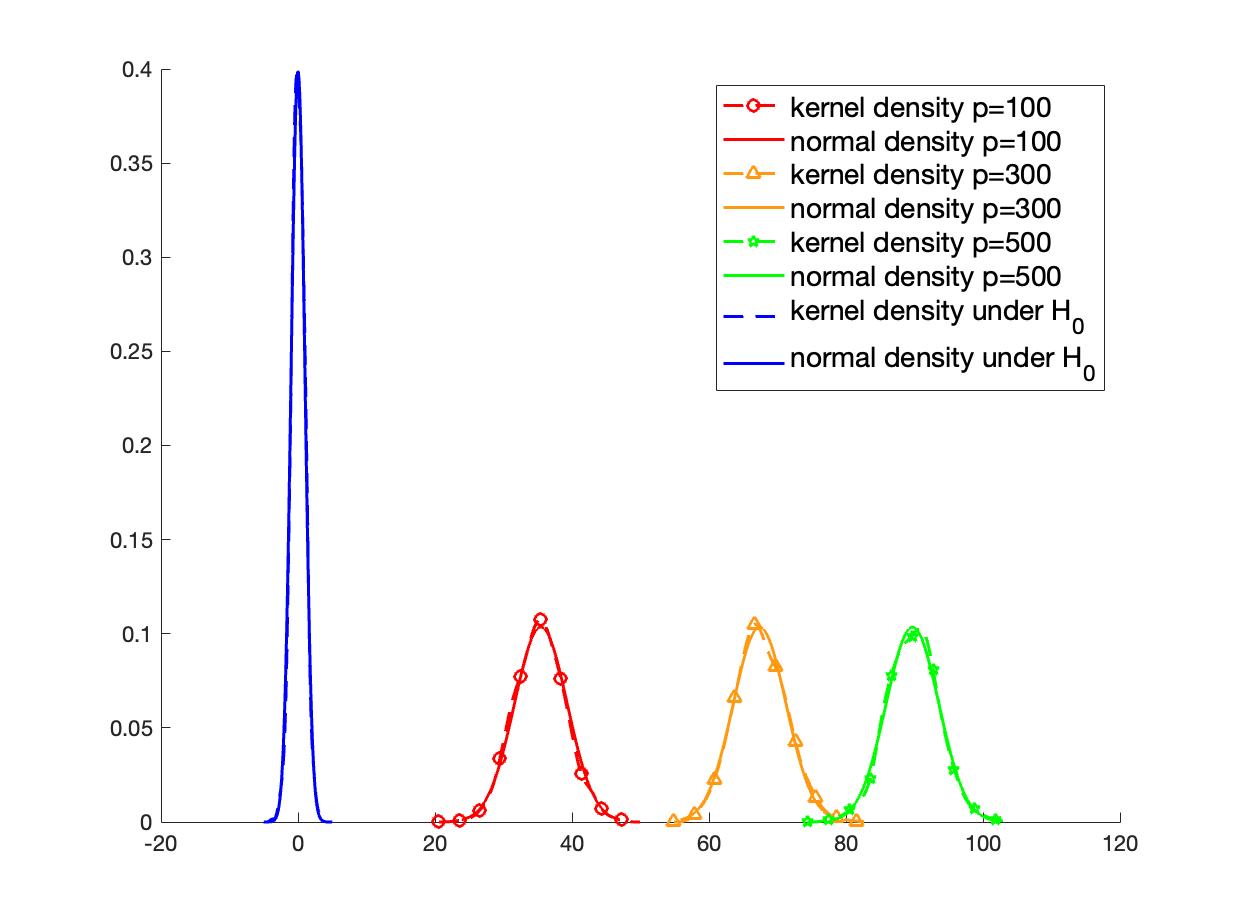}
		\captionsetup{font={scriptsize}}
		\caption{Case 2}
	\end{minipage}
	\begin{minipage}[t]{0.3\textwidth}
		\centering
		\includegraphics[width=4.5cm]{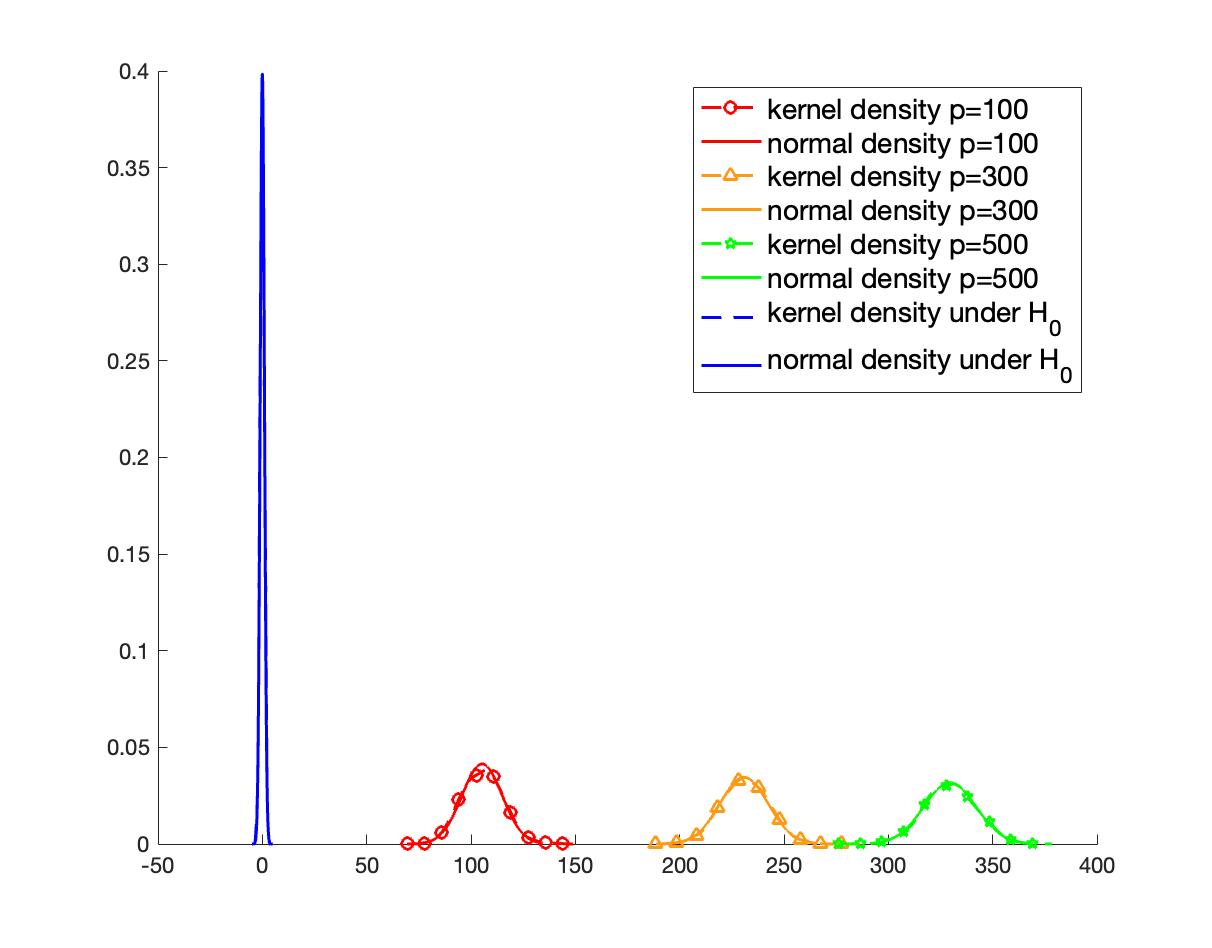}
		\captionsetup{font={scriptsize}}
		\caption{Case 3}
	\end{minipage}
	
	\begin{minipage}[t]{0.3\textwidth}
		\centering
		\includegraphics[width=4.5cm]{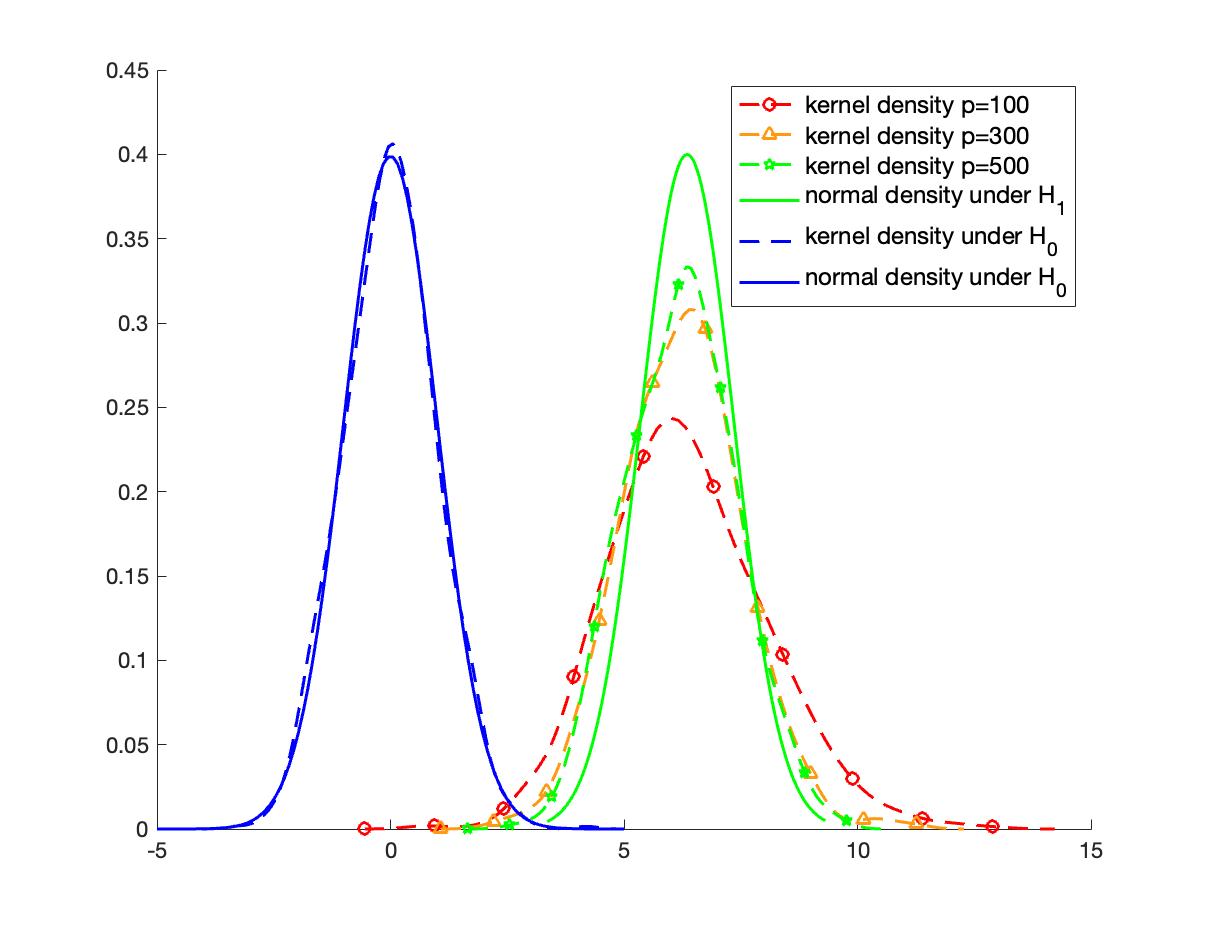}
		\captionsetup{font={scriptsize}}
		\caption{Case 4}
	\end{minipage}
	\begin{minipage}[t]{0.3\textwidth}
		\centering
		\includegraphics[width=4.5cm]{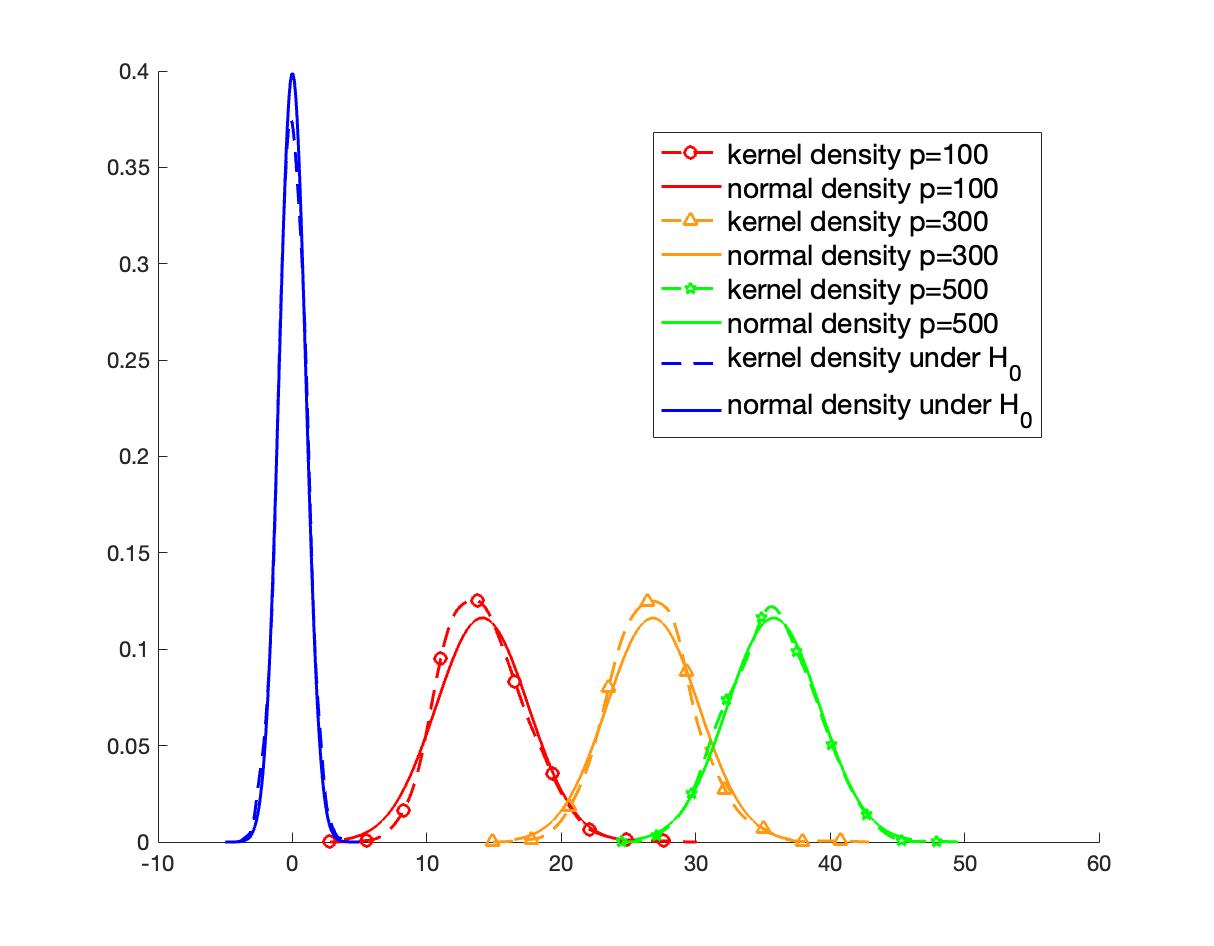}
		\captionsetup{font={scriptsize}}
		\caption{Case 5}
	\end{minipage}
	\begin{minipage}[t]{0.3\textwidth}
		\centering
		\includegraphics[width=4.5cm]{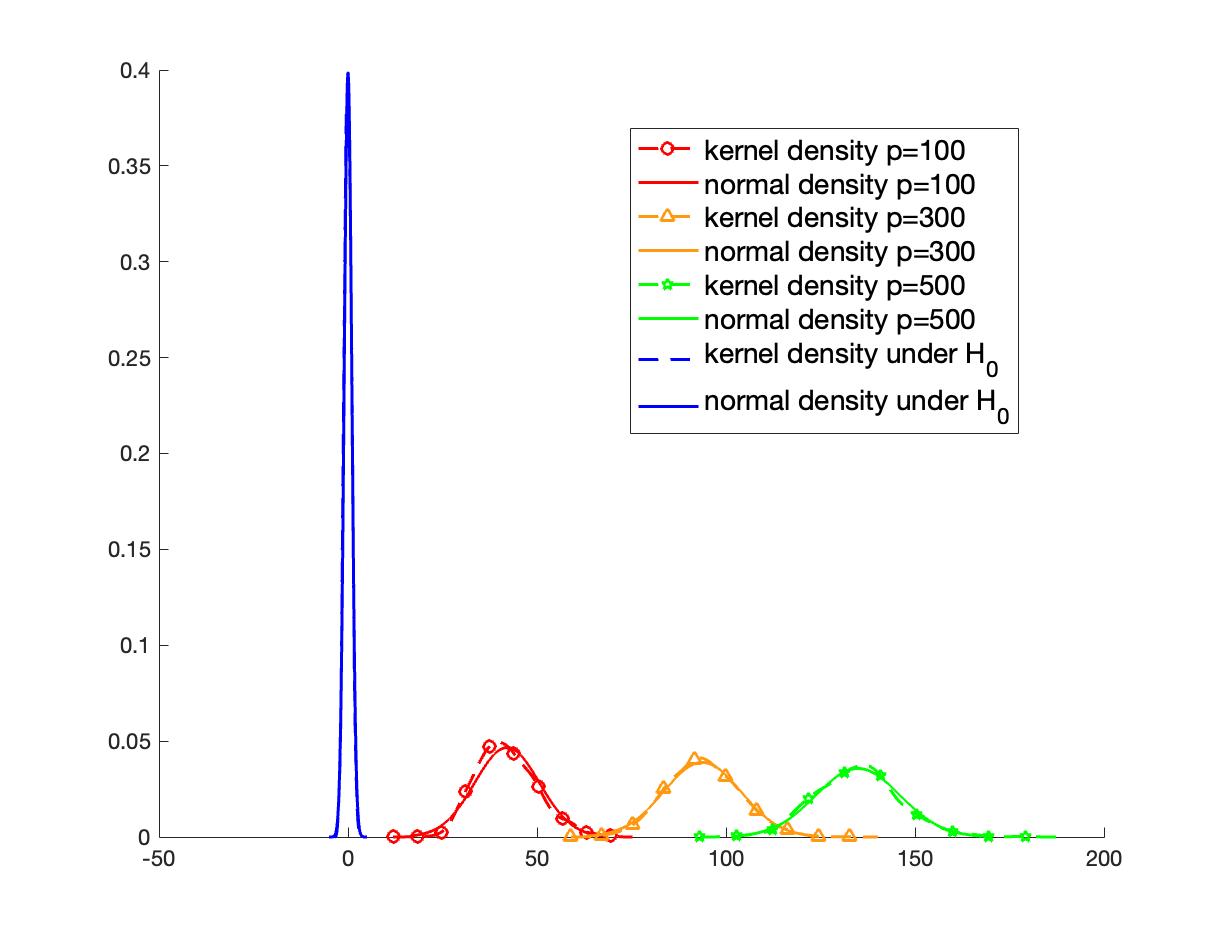}
		\captionsetup{font={scriptsize}}
		\caption{Case 6}
	\end{minipage}
	\caption*{Comparisons of kernel estimated density functions and their limits}

\end{figure}

\section{Technical proofs}\label{section 6}
In this section, we present the proofs of Theorems \ref{thm1}, \ref{thm2}, \ref{thm3} and \ref{thm4}. First, we truncate and renormalize the random variables to ensure the existence of their higher order moments. 
\subsection{ Truncation and renormalization}
Let $\hat{x}_{i j}=x_{i j} \mathbbm{1}_{\left\lbrace \left|x_{i j}\right|<\eta_{n} \sqrt{n}\right\rbrace} $ and $\tilde{x}_{i j}=\frac{\hat{x}_{i j}-\mathbbm{E} \hat{x}_{i j}}{ \hat\sigma_{n}}$, where $\hat\sigma_{n}^{2}=\mathbbm{E}\left|\hat{x}_{i j}-\mathbbm{E} \hat{x}_{i j}\right|^{2}$ and $\eta_{n} \rightarrow 0$ with a slow rate. Correspondingly, define $ \hat{\bbB}=\frac{1}{n}\bbT\hat{\bbX}\hat{\bbX}^{\ast}\bbT^{\ast} $ and $ \tilde{\bbB}=\frac{1}{n}\bbT\tilde{\bbX}\tilde{\bbX}^{\ast}\bbT^{\ast} $, where $ \hat{\bbX}=(\hat{x}_{i j})$ and $ \tilde{\bbX}=(\tilde{x}_{i j})$. $ \hat{G}_{n} $ and $ \tilde{G}_{n} $  denote the analogues of $ G_{n} $ with the matrix $ \bbB $ replaced by $ \hat{\bbB} $ and $ \tilde{\bbB} $, respectively. 
Next, we demonstrate that the entries of $\mathbf{X}$ in the LSSs are equivalent to and can be replaced by the truncated and renormalized entries.

According to the Lindeberg-type condition in Assumption \ref{ass1}, we obtain that as $\min\{n, p\} \rightarrow \infty$, 
\begin{align*}
		&\mathrm{P}(\bbB \neq \hat{\bbB})
		\leq \sum_{i, j} \mathrm{P}\left(\left|x_{i j}\right| \geq \eta_{n} \sqrt{n}\right)\rightarrow 0.
	\end{align*}
It follows from the definition of LSSs that for any $ l=1,\ldots,h $, 
\begin{align}  
&\left| \int f_{l}(x)d\hat{G}_{n}-\int f_{l}(x)d\tilde{G}_{n}(x)\right| =\sum_{i=1}^{p}\left| f_{l}(\lambda_{i}^{\hat{\bbB}})-f_{l}(\lambda_{i}^{\tilde{\bbB}})\right| \nonumber\\
\leq&\sum_{i=1}^{M}\left| f_{l}(\lambda_{i}^{\hat{\bbB}})-f_{l}(\lambda_{i}^{\tilde{\bbB}})\right|+\sum_{i=M+1}^{p}\left| f_{l}(\lambda_{i}^{\hat{\bbB}})-f_{l}(\lambda_{i}^{\tilde{\bbB}})\right|. \label{htGn}
\end{align}
Using the same discussion in \cite{10.1214/aop/1078415845}, we can easily obtain that the second term of \eqref{htGn} tends to 0 in probability. 
For the first term of \eqref{htGn}, from the arguments in Supplement B of \cite{JiangB21G}, we know that
\begin{align} \left|\lambda_{i}^{\hat{\bbB}}-\lambda_{i}^{\tilde{\bbB}}\right|=o_{p}(n^{-\frac{1}{2} } {\rho_{i}}) .\label{truncation}
\end{align} 
Then, for brevity, we denote $ \beta_{i}=({\lambda_{i}^{\hat{\bbB}}-\lambda_{i}^{\tilde{\bbB}}})/{\rho_{i}}$ and
 obtain that
\begin{align*}
	&f_l(\lambda_{i}^{\hat{\bbB}})-f_l(\lambda_{i}^{\tilde{\bbB}})	=\int_{0}^{\beta_{i}\rho_{i}}f_l^{\prime}(t+\lambda_{i}^{\tilde{\bbB}})	dt\\
	=&\int_{0}^{1}\beta_{i}\rho_{i}f_l^{\prime}(\beta_{i}\rho_{i}s+\lambda_{i}^{\tilde{\bbB}})ds
	=\beta_{i}\rho_{i}f_l^{\prime}(\rho_{i})\int_{0}^{1}\frac{f_l^{\prime}(\rho_{i}(\beta_{i}s+\frac{\lambda_{i}^{\tilde{\bbB}}}{\rho_{i}}) )}{f_l^{\prime}(\rho_{i})}ds,
\end{align*} 
which, together with Assumption \ref{ass4} and (\ref{truncation}), implies
 $$  \sqrt{n}\sum_{i=1}^{M}\dfrac{\left| f_{l}(\lambda_{i}^{\hat{\bbB}})-f_{l}(\lambda_{i}^{\tilde{\bbB}})\right|}{f_{l}^{\prime}(\rho_{i})\rho_{i}}= o_{p}(1).$$
Therefore, in the following proofs, we can safely assume that $ \left|x_{ij} \right|<\eta_{n}\sqrt{n}$. 

\subsection{Some primary definitions and lemmas}
In this section, we provide some useful results that will be used later in the proofs of Theorems \ref{thm1} and  \ref{thm2}. For the population covariance matrix $\bSi=\bbT\bbT^{\ast}$, consider the corresponding sample covariance matrix $ \bbB=\bbT\bbS_{x}\bbT^{\ast} $, where $ \bbS_{x}=\frac{1}{n}\bbX\bbX^{\ast}  $. 
By singular value decomposition of $ \bbT $ (see \eqref{decT}),
\begin{equation*}
	\bbB=\bbV\left(\begin{array}{cc}
		\bbD_{1}^\frac{1}{2}\bbU_{1}^{\ast}\bbS_{x}\bbU_{1}\bbD_{1}^\frac{1}{2}  &\bbD_{1}^\frac{1}{2}\bbU_{1}^{\ast}\bbS_{x}\bbU_{2}\bbD_{2}^\frac{1}{2}\\
		\bbD_{2}^\frac{1}{2}\bbU_{2}^{\ast}\bbS_{x}\bbU_{1}\bbD_{1}^\frac{1}{2}	&\bbD_{2}^\frac{1}{2}\bbU_{2}^{\ast}\bbS_{x}\bbU_{2}\bbD_{2}^\frac{1}{2}
	\end{array} \right)\bbV^{\ast}.
\end{equation*}
Note that
\begin{equation*}
	\bbS=\left(\begin{array}{cc}
		\bbD_{1}^\frac{1}{2}\bbU_{1}^{\ast}\bbS_{x}\bbU_{1}\bbD_{1}^\frac{1}{2}  &\bbD_{1}^\frac{1}{2}\bbU_{1}^{\ast}\bbS_{x}\bbU_{2}\bbD_{2}^\frac{1}{2}\\
		\bbD_{2}^\frac{1}{2}\bbU_{2}^{\ast}\bbS_{x}\bbU_{1}\bbD_{1}^\frac{1}{2}	&\bbD_{2}^\frac{1}{2}\bbU_{2}^{\ast}\bbS_{x}\bbU_{2}\bbD_{2}^\frac{1}{2}
	\end{array} \right)	\triangleq\left(\begin{array}{cc}\bbS_{11}&\bbS_{12}\\\bbS_{21}&\bbS_{22}\end{array} \right). 
\end{equation*}
Note that $ \bbB $ and $ \bbS $ have the same eigenvalues.

Recall that $\underline{\bbB}=\frac{1}{n}\bbX^{\ast}\bbT^{\ast}\bbT\bbX $ (the spectral of which differs from that of $ \bbB $ by $ \left|  n-p\right|  $ zeros). Its limiting spectral distribution is $ \underline{F}^{c,H} $, $\underline{F}^{c,H}\equiv\left(1-c \right)\mathbbm{1}_{\left[0,\infty \right) }+cF^{c,H}   $, and its Stieltjes transform is $ \underline{m}\left(z \right)  $.
Let $ \widetilde{\bgl}_{j} $ be the eigenvalues of $ \bbS_{22} $ so that the LSS of $ \bbS_{22} $ is $ \sum_{j=1}^{p-M}f( \widetilde{\bgl}_{j}) $. Correspondingly, recall that $ c_{nM}:=\frac{p-M}{n} $, $ H_{2n}:=F^{\bbD_{2}} $ and $ m_{2n0}:=m_{2n0}(z) $ are the 
Stieltjes transforms
 of $ F^{c_{nM},H_{2n}} $. Then, we have the following preliminary results:


\begin{lemma}\label{lemma1}
	Under Assumptions \ref{ass1}-\ref{ass4},  $$ \left( p-M\right)\int f\left(x \right)dF^{c_{nM},H_{2n}} =p\int f\left( x\right)dF^{c_{n},H_{n}}\left( x\right). $$
	
\end{lemma}
\begin{proof}
	By the Cauchy integral formula, 
	$$ p\int f(x)dF^{c_{n},H_{n}}=-\frac{p}{2 \pi i}\oint_\mathcal{C}f(z)m_{1n0}dz=-\frac{n}{2 \pi i}\oint_\mathcal{C}f(z)\underline{m}_{1n0}dz, $$ $$ (p-M)\int f(x)dF^{c_{nM},H_{2n}}=-\frac{p-M}{2 \pi i}\oint_\mathcal{C}f(z)m_{2n0}dz=-\frac{n}{2 \pi i}\oint_\mathcal{C}f(z)\underline{m}_{2n0}dz, $$
	where $\underline{m}_{1n0} $ and $ \underline{m}_{2n0}$ are the Stieltjes transforms of $\underline{F}^{c_{n},H_{n}} $ and $ \underline{F}^{c_{nM},H_{2n}}$, respectively.
	Then,  $$ (p-M)\int f(x)dF^{c_{nM},H_{2n}}-p\int f(x)dF^{c_{n},H_{n}}=\frac{n}{2 \pi i}\oint_\mathcal{C}f(z)\left(\underline{m}_{1n0}-\underline{m}_{2n0}\right) dz. $$
	Next, we prove that $\underline{m}_{1n0}=\underline{m}_{2n0}$.
	
	Note that
	$ m_{1n0} $ and $ m_{2n0} $ are the unique solutions to
	\begin{align}
		z=-\frac{1}{\underline{m}_{1n0}}+c_{n}\int\frac{tdH_{n}\left( t\right) }{1+t\underline{m}_{1n0}}\label{1}\\     
		z=-\frac{1}{\underline{m}_{2n0}}+c_{nM}\int\frac{tdH_{2n}\left( t\right) }{1+t\underline{m}_{2n0}} \label{2}, 
	\end{align}
	respectively,
	where $ \underline{m}_{1n0}=-\frac{1-c_{n}}{z}+c_{n}m_{1n0} $ and $ \underline{m}_{2n0}=-\frac{1-c_{nM}}{z}+c_{nM}m_{2n0} $. 
	Since $$  H_{n}(t)=\frac{1}{p}\left[\sum_{i=1}^{M}\mathbbm{1}_{\left\lbrace 0\leq t\right\rbrace }+\sum_{i=M+1}^{p}\mathbbm{1}_{\left\lbrace \alpha_{i}\leq t\right\rbrace } \right]=\frac{M}{p}+\frac{1}{p}\sum_{i=M+1}^{p}\mathbbm{1}_{\left\lbrace \alpha_{i}\leq t\right\rbrace }  $$ and $ H_{2n}(t)=\frac{1}{p-M}\sum_{i=M+1}^{p}\mathbbm{1}_{\left\lbrace\alpha_{i}\leq t \right\rbrace } $,
	  (\ref{1}) can be written as
	\begin{align}
		\nonumber z&=-\frac{1}{\underline{m}_{1n0}}+\frac{p}{n}\int\frac{td\left( \frac{M}{p}+\frac{1}{p}\sum_{i=M+1}^{p}\mathbbm{1}_{\left\lbrace \alpha_{i}\leq t\right\rbrace }\left( t\right)\right)  }{1+t\underline{m}_{1n0}}\\	
		\nonumber&=-\frac{1}{\underline{m}_{1n0}}+\frac{p}{n}\int\frac{td\left( \frac{1}{p}\sum_{i=M+1}^{p}\mathbbm{1}_{\left\lbrace \alpha_{i}\leq t\right\rbrace }\left( t\right)\right)  }{1+t\underline{m}_{1n0}}\\
		&=-\frac{1}{\underline{m}_{1n0}}+\frac{1}{n}\sum_{i=M+1}^{p}\frac{\alpha_{i}}{1+\al_{i}\underline{m}_{1n0}}.\label{3}
	\end{align}	
	Similarly, equation (\ref{2}) can be written as 
	\begin{align}
		z=-\frac{1}{\underline{m}_{2n0}}+\frac{1}{n}\sum_{i=M+1}^{p}\frac{\al_{i}}{1+\al_{i}\underline{m}_{2n0}}.\label{4}
	\end{align}	
Thus, according to the fact that $ m_{1n0} $ and $ m_{2n0} $ are the unique solutions of \eqref{3} and \eqref{4}, respectively, we have $ m_{1n0} = m_{2n0} $, which completes the proof of this lemma.
\end{proof}

\begin{lemma}\label{lemma2} \textit{Under Assumptions \ref{ass1}-\ref{ass4}}, 
$$\sum_{j=M+1}^{p}f\left( \lambda_{j}\right)   - \sum_{j=1}^{p-M}f\left( \widetilde{\bgl}_{j}\right)- \frac{M}{2\pi i}\oint_\mathcal{C}f\left(z \right)\frac{\underline{m}^{'}(z)}{\underline{m}(z)}dz=o_p(1).$$
\end{lemma}
\begin{proof}
Note that
	\begin{align*}
		L_{1}:=\sum_{j=M+1}^{p}f\left( \lambda_{j}\right).
	\end{align*}
	By the Cauchy integral formula, we have
	\begin{align*}
		L_{1}=-\frac{p}{2\pi i}\oint_{\mathcal C}f\left(z \right)m_{n}\left(z \right)dz,
	\end{align*}
	where $m_{n}=\frac{1}{p}\mathrm{tr}\left(\bbS-z\bbI_{p} \right)^{-1}=\frac{1}{p}\mathrm{tr}\left(\bbB-z\bbI_{p} \right)^{-1} $. Analogously, we have
	\begin{align*}
		L_{2}:=\sum_{j=1}^{p-M}f\left( \widetilde{\bgl}_{j}\right)=-\frac{p-M}{2\pi i}\oint_{\mathcal C}f\left(z \right)m_{2n}\left(z \right)dz,	
	\end{align*}
	where $m_{2n}=\frac{1}{p-M}\mathrm{tr}\left(\bbS_{22}-z\bbI_{p-M} \right)^{-1}$. By applying the block matrix inversion formula to $m_n$, we can obtain
	\begin{align}\label{L12}
		L_{1}-L_{2}=-\frac{1}{2\pi i}\oint_{\mathcal C}f\left(z \right)\left(T_{1}-T_{2} \right) dz,
	\end{align}
	where 
	\begin{align*}
		T_{1} &=\mathrm{tr}\left( \bbS_{11}-z\bbI_{M}-\bbS_{12}\left( \bbS_{22}-z\bbI_{p-M}\right)^{-1}\bbS_{21} \right)^{-1},\\ 	
		T_{2} &=-\mathrm{tr}\left[\left( \bbS_{11}-z\bbI_{M}-\bbS_{12}\left(\bbS_{22}-z\bbI_{p-M} \right)^{-1}\bbS_{21} \right)^{-1}\bbS_{12}\left(\bbS_{22}-z\bbI_{p-M} \right)^{-2}\bbS_{21}\right].   
	\end{align*}
	Note that for any matrix $\bbZ$, 
	\begin{align*}
		\bbZ\left(\bbZ^{\ast}\bbZ-\lambda \bbI \right)^{-1}\bbZ^{\ast}=\bbI+\lambda\left(\bbZ\bbZ^{\ast}-\lambda \bbI \right)^{-1},
	\end{align*}	
	which, together with the notation $\bUps_n:=\frac{1}{n}\bbD_{1}^{\frac{1}{2}}\bbU_{1}^{\ast}\bbX\left(\frac{1}{n}\bbX^{\ast}\bbU_{2}\bbD_{2}\bbU_{2}^{*}\bbX-z\bbI_{n}\right) ^{-1}  \bbX^{\ast}\bbU_{1}\bbD_{1}^{\frac{1}{2}} $, implies that
	\begin{align*}
		 T_{1}&=-z^{-1}\mathrm{tr}\left(\bbI_{M}+\bUps_n \right)^{-1}\\
		T_{2} &=z^{-1}\mathrm{tr}\left[\left(\bbI_{M}+\bUps_n \right)^{-1}\bbS_{12}\left(\bbS_{22}-z\bbI_{p-M} \right)^{-2}\bbS_{21}\right]
	\end{align*}
	$ \underline{m}_{2n}= \underline{m}_{2n}(z)$ denotes the Stieltjes transform of $ F^{\frac{1}{n}\bbX^{\ast}\bbU_{2}\bbD_{2}\bbU_{2}^{*}\bbX} $. Thus,  we have that $\underline{m}_{2n}(z)-\underline{m}(z)=o_p(1)$ for any $z\in\mathcal{C}$. 
	From Theorem 3.1 of \citep{JiangB21G}, we know that
	\begin{align} 
		\frac{1}{n}\bbU_{1}^{\ast}\bbX\left(\frac{1}{n}\bbX^{\ast}\bbU_{2}\bbD_{2}\bbU_{2}^{*}\bbX-z\bbI_{n}\right) ^{-1}\bbX^{\ast}\bbU_{1}
		=\underline{m}_{2n}\left(z \right)\bbI_{M}+O_{p}(n^{-\frac{1}{2}}).
		\label{6}
	\end{align}
Thus, under Assumption \ref{ass3}, we obtain that
	\begin{align} 
	\bbD_{1}^{1/2}\left(\bbI_{M}+\bUps_n \right)^{-1}\bbD_{1}^{1/2}=\frac{1}{\underline{m} \left(z \right)}\bbI_{M}+o_p(1),
	\label{7}
	\end{align}
 which yields
 \begin{align}\label{T1op1}
  T_{1}=o_p(1).
\end{align}

	It follows that
	\begin{align} 
	&\bbD_{1}^{-1/2}\bbS_{12}\left(\bbS_{22}-z\bbI_{p-M} \right)^{-2}\bbS_{21}\bbD_{1}^{-1/2}\nonumber\\
		\nonumber=&\frac{1}{n}\mathrm{tr}\left[ \left(\bbS_{22}-z\bbI_{p-M} \right)^{-2}\bbS_{22}\right]\bbI_{M}+O_{p}(n^{-\frac{1}{2}}) \\
		\nonumber=&\frac{1}{n}\mathrm{tr}\left(\bbS_{22}-z\bbI_{p-M} \right)^{-1}\bbI_{M}+\frac{z}{n}\mathrm{tr}\left(\bbS_{22}-z\bbI_{p-M} \right)^{-2}\bbI_{M}+O_{p}(n^{-\frac{1}{2}})\\
		=& cm\left(z \right)\bbI_{M}+zcm'\left(z \right)\bbI_{M}+o_p(1)\nonumber\\
		=&\underline{m}\left(z \right)\bbI_{M}+z\underline{m}'\left(z \right)\bbI_{M}+o_p(1), \label{8} 
	\end{align}
	where the last equality is derived from $ \underline{m}=-\frac{1-c}{z}+cm\left(z \right)  $. Therefore, according to (\ref{7}) and (\ref{8}), we obtain
	\begin{align*}
		T_{2}=M\frac{\underline{m}\left(z \right)+z\underline{m}'\left(z \right)}{z\underline{m}\left(z \right)}+o_p(1),
	\end{align*}
	which, together with \eqref{L12} and \eqref{T1op1}, implies that
	\begin{align*}
		L_{1}-L_{2}=\frac{M}{2\pi i}\oint_{\mathcal C}f\left(z \right)\frac{\underline{m}(z)+z\underline{m}^{'}(z)}{z\underline{m}(z)}dz+o_p(1)=\frac{M}{2\pi i}\oint_{\mathcal C}f\left(z \right)\frac{\underline{m}^{'}(z)}{\underline{m}(z)}dz+o_p(1).
	\end{align*}
	Therefore, the proof of this lemma is complete.
\end{proof}
Define random vector $\boldsymbol\gamma_{k}=\left( \gamma_{kj} \right)' = \left( \sqrt{n} \frac{\lambda_{j}-\phi_{n}\left(\al_{k} \right) }{\phi_{n}\left(\al_{k} \right)},    j\in J_{k} \right)'$, where $ J_{k} $ is the indicator set of a packet of $ d_{k} $ consecutive sample eigenvalues. Then, we present the following lemma, which is borrowed from \cite{JiangB21G} and characterizes the limiting distribution of the spiked eigenvalues of the sample covariance matrix. 
\begin{lemma} (\cite{JiangB21G})\label{lemma4}
	 Under Assumptions \ref{ass1}-\ref{ass4}, \textit{ random vector} $ \boldsymbol\gamma_{k} $ \textit{converges weakly to the joint distribution of $ d_{k} $ eigenvalues of a Gaussian random matrix} $$ -\frac{1}{\theta_{k}}\left[\bgO_{\phi_{k}} \right]_{kk},   $$
	\textit{where} 
	$$\theta_{k}=\phi_{k}^{2}\underline{m}_{2}\left( \phi_{k}\right), ~~\underline{m}_{2}\left( \lambda\right)=\int\frac{1}{\left( \lambda-x\right) ^{2}}d\underline{F}^{c,H}\left( x\right)   $$ \textit{with} $ \underline{F}^{c,H} $ \textit{being the} LSD \textit{of matrix} $ n^{-1}\bbX^{\ast}\bbU_{2}\bbD_{2}\bbU_{2}^{*}\bbX,$ $ \phi_{k}=\al_{k}\left(1+c\int\frac{t}{\al_{k}-t}dH\left(t \right)  \right)  $.
	$ \left[\bgO_{\phi_{k}} \right]_{kk} $ \textit{is the} $ k $\textit{th diagonal block of matrix} $ \bgO_{\phi_{k}} $. \textit{The variances and covariances of the elements} $ \omega_{ij} $ \textit{of} $ \bgO_{\phi_{k}} $ \textit{are:}
	$$
	\operatorname{Cov}\left(\omega_{i_{1}, j_{1}}, \omega_{i_{2}, j_{2}}\right)=\left\{\begin{array}{cc}
		(\al_{x}+1) \theta_{k}+\beta_{x}\mathcal{U}_{ i i i i} \nu_{k}, & i_{1}=j_{1}=i_{2}=j_{2}=i \\
		\theta_{k}+\beta_{x}\mathcal{U}_{ i j i j} \nu_{k}, & i_{1}=i_{2}=i \neq j_{1}=j_{2}=j \\
		\beta_{x}\mathcal{U}_{ i_{1} j_{1} i_{2} j_{2}} \nu_{k}, & \text { other cases }
	\end{array}\right.
	$$
	\textit{where} $\beta_{x}\mathcal{U}_{ i_{1} j_{1} i_{2} j_{2}}=\sum_{t=1}^{p} \bar{u}_{t i_{1}} u_{t j_{1}} u_{t i_{2}} \bar{u}_{t j_{2}} \beta_{x}$, $\boldsymbol u_{i}=\left(u_{1 i}, \ldots, u_{p i}\right)^{\prime}$ \textit{are the} $i$ \textit{th column of the matrix} $\mathbf{U}_{1}$, $\nu_{k}=\phi_{k}^{2} \underline{m}^{2}\left(\phi_{k}\right)$.
\end{lemma}

Recall that $ \lambda_{j} $ is the eigenvalue of $ \bbB $, and $ \widetilde{\bgl}_{j} $ is the eigenvalue of $ \bbS_{22} $. The following lemma shows the independence between $ \sum_{j=1}^{M}f\left( \lambda_{j}\right) $ and $ \sum_{j=1}^{p-M}f\left(\widetilde{\bgl}_{j} \right) $.

\begin{lemma}\label{lemma5} Under Assumptions \ref{ass1}-\ref{ass4},  $ \sum_{j=1}^{M}f\left( \lambda_{j}\right) $ and $ \sum_{j=1}^{p-M}f\left(\widetilde{\bgl}_{j} \right) $  are asymptotically independent.
\end{lemma}
\begin{proof} 
	It is sufficient to prove that for a given $ \sum_{j=1}^{p-M}f\left( \widetilde{\lambda}_{j}\right) $, the asymptotic limiting distribution of $ \sum_{j=1}^{M}f\left( \lambda_{j}\right) $ does not depend on the random part of $ \sum_{j=1}^{p-M}f\left( \widetilde{\lambda}_{j}\right) $, that is, it only depends on its limit. 
	
	First, we consider $ f(x)=x $.
	From the proof of Theorem 3.1 in \cite{JiangB21G}, we have that
	$$ 0=\left| \left[\bgO_{M}\left( \phi_{k}\right) \right] _{kk}+\mathrm{ lim}\gamma_{kj}\left\lbrace \phi_{k}^{2}\underline m_{2}(\phi_{k})\right\rbrace \bbI_{d_{k}}\right|,$$ 
	where $ \bgO_{M}\left( \phi_{k}\right) $ 
	\begin{align*}
=\frac{\phi_{k}}{\sqrt{n}}\left[ \mathrm{tr}\left\lbrace \left(  \phi_{k}\bbI-\frac{1}{n}\bbX^{*}\bbU_{2}\bbD_{2}\bbU_{2}^{*} \bbX\right) ^{-1}\right\rbrace \bbI-\bbU_{1}^{*}\bbX\left(\phi_{k}\bbI-\frac{1}{n}\bbX^{*}\bbU_{2}\bbD_{2}\bbU_{2}^{*} \bbX\right) ^{-1}\bbX^{*}\bbU_{1}\right],
\end{align*}
and
 $\underline m_{2}(\phi_{k}) $ is the limit of $ \mathrm{tr} \left(  \phi_{k}\bbI-\frac{1}{n}\bbX^{*}\bbU_{2}\bbD_{2}\bbU_{2}^{*} \bbX\right) ^{-2}$. Then, we know that $ \gamma_{kj} $ has the same asymptotic distribution with eigenvalues of
	$ -\frac{\left[\bgO_{M}\left( \phi_{k}\right) \right] _{kk}}{\phi_{k}^{2}\underline m_{2}(\phi_{k})} $ in order. Given $ \bbU_{2}^{*}\bbX $, from \cite{JiangB21G}, we could suppose that $ \bbU_{2}^{*} \bbX $ and $  \bbU_{1}^{*} \bbX $ are independent. Then, the limiting distribution of $ \gamma_{kj}  $ only depends on the limit of $\mathrm{tr}\left(  \phi_{k}\bbI-\frac{1}{n}\bbX^{*}\bbU_{2}\bbD_{2}\bbU_{2}^{*} \bbX\right) ^{-1} $, that is, $ \underline{m}_{2}(\phi_{k}) $, and has nothing to do with the random part. Therefore, the independence between $ \sum_{j=1}^{M}f\left( \lambda_{j}\right) $  and $ \sum_{j=1}^{p-M}f\left( \widetilde{\lambda}_{j}\right) $ is obtained when $ f(x)=x. $ 
	
	When $ f(x)\neq x $, 
	by using the Newton-Leibniz formula, we have 
	\begin{align}
		&\nonumber\sum_{j=1}^{M}f\left(\lambda_{j}\right)-\sum_{k=1}^{K}d_{k}f\left(\phi_{n}\left(\bbalp_{k} \right) \right)
		=\nonumber\sum_{k=1}^{K}\sum_{ j\in J_{k}}(f\left(\lambda_{j}\right) - f\left(\phi_{n}\left(\bbalp_{k} \right) \right))  \\
		&=\nonumber\sum_{k=1}^{K}\sum_{j\in J_{k}}\int_{0}^{\frac{\phi_{n}\left(\bbalp_{k} \right)}{\sqrt{n}}\gamma_{kj}}f'\left(t+\phi_{n}\left(\bbalp_{k} \right) \right)dt\\
		&=\nonumber\sum_{k=1}^{K}\sum_{j\in J_{k}}\int_{0}^{1}\frac{\phi_{n}\left(\bbalp_{k} \right)}{\sqrt{n}}\gamma_{kj}\frac{f'\left(\phi_{n}\left(\bbalp_{k} \right)\left(1+\frac{\gamma_{kj}}{\sqrt{n}}s\right)  \right)}{f'\left(\phi_{n}\left(\bbalp_{k} \right)\right) }f'\left( \phi_{n}\left(\bbalp_{k} \right)\right)ds\\
		&\rightarrow\sum_{k=1}^{K}\sum_{j\in J_{k}}\int_{0}^{1}\gamma_{kj}\varpi^{k} ds
		=\sum_{k=1}^{K}\sum_{j\in J_{k}}\varpi^{k}\gamma_{kj},    \label{21}
	\end{align}
	where \eqref{21} is true due to Assumption \ref{ass4}, and $ \varpi^{k}=\lim\dfrac{\phi_{n}(\al_{k})}{\sqrt{n}}f^{\prime}(\phi_{n}(\al_{k}) )$. Thus, we turn it into a function of $ \gamma_{kj} $. Since we have proven above that given $ \sum_{j=1}^{p-M}f\left(\widetilde{\lambda}_{j}\right) $, the limiting distribution of $ \gamma_{kj} $ is only concerned with the limit of $ \sum_{j=1}^{p-M}f\left( \widetilde{\lambda}_{j}\right) $, as is $ \sum_{k=1}^{K}\sum_{j\in J_{k}}\varpi^{k}\gamma_{kj} $, accordingly, we can conclude that $ \sum_{j=1}^{M}f\left( \lambda_{j}\right) $  and $ \sum_{j=1}^{p-M}f\left( \widetilde{\lambda}_{j}\right) $ are asymptotically independent. The proof is complete.
\end{proof}

The following lemma derives the asymptotic distribution of the LSS generated from submatrix $ \bbS_{22} $.}

\begin{lemma}\label{lemma3}
	Define $ Q_{1}=\sum_{j=1}^{p-M}f_{1}(\widetilde{\lambda}_{j}) -(p-M)\int f_{1}(x)dF^{c_{nM},H_{2n}}$; then, under Assumptions \ref{ass1}-\ref{ass4}, we have $$ \kappa_{1}^{-1}\left(Q_{1}-\mu_{1} \right)\stackrel{d}\longrightarrow N\left(0,1 \right)  $$
	with mean function
	\begin{align*}
	\mu_{1}&=-\frac{\alpha_{x}}{2 \pi i}\cdot\oint_{\mathcal{C}}f_{1}(z)\frac{  c_{nM} \int \underline{m}_{2n0}^{3}(z)t^{2}\left(1+t \underline{m}_{2n0}(z)\right)^{-3} d H_{2n}(t)}{\left(1-c_{nM} \int \frac{\underline{m}_{2n0}^{2}(z) t^{2}}{\left(1+t \underline{m}_{2n0}(z)\right)^{2}} d H_{2n}(t)\right)\left(1-\alpha_{x} c_{nM} \int \frac{\underline{m}_{2n0}^{2}(z) t^{2}}{\left(1+t \underline{m}_{2n0}(z)\right)^{2}} d H_{2n}(t)\right) }dz \\
		&-\frac{\beta_{x}}{2 \pi i} \cdot \oint_{\mathcal{C}} f_{1}(z) \frac{c_{nM} \int \underline{m}_{2n0}^{3}(z) t^{2}\left(1+t \underline{m}_{2n0}(z)\right)^{-3} d H_{2n}(t)}{1-c_{nM} \int \underline{m}_{2n0}^{2}(z) t^{2}\left(1+t \underline{m}_{2n0}(z)\right)^{-2} d H_{2n}(t)} dz,  	
	\end{align*}
	and the covariance function is   
	\begin{align*}
		\kappa_{1}^{2}=-\frac{1}{4\pi^{2}}\oint_{\mathcal{C}_{1}}\oint_{\mathcal{C}_{2}}f_{1}\left(z_{1} \right)f_{1}\left(z_{2} \right)\vartheta_{n}^{2}dz_{1}dz_{2},
	\end{align*}
where 
$\vartheta_{n}^{2}=\Theta_{0,n}(z_{1},z_{2})+\al_{x}\Theta_{1,n}(z_{1},z_{2})+\beta_{x}\Theta_{2,n}(z_{1},z_{2}),$ 
\begin{align*}
	\Theta_{0,n}(z_{1},z_{2})&=\dfrac{\underline{m}_{2n0}^{\prime}(z_{1}) \underline{m}_{2n0}^{\prime}(z_{2})  }{(\underline{m}_{2n0}(z_{1})-\underline{m}_{2n0}(z_{2}))^{2}  }-\dfrac{1}{(z_{1}-z_{2})^{2}},\\
	\Theta_{1,n}(z_{1},z_{2})&=\frac{\partial}{\partial z_{2}}\left\lbrace \dfrac{\partial \mathcal{A}_{n}(z_{1},z_{2})}{\partial z_{1}}\dfrac{1}{1-\al_{x}\mathcal{A}_{n}(z_{1},z_{2})} \right\rbrace ,\\
	\mathcal{A}_{n}(z_{1},z_{2})&=\dfrac{z_{1}z_{2}}{n}\underline{m}_{2n0}^{\prime}(z_{1}) \underline{m}_{2n0}^{\prime}(z_{2})\mathrm{tr}{\bGma^{*}\bbP_{n}(z_{1})\bGma\bGma^{\prime}\bbP_{n}(z_{2})^{\prime} \bar{\bGma}},\\
	\Theta_{2,n}(z_{1},z_{2})&=\dfrac{z_{1}^{2}z_{2}^{2}\underline{m}_{2n0}^{\prime}(z_{1}) \underline{m}_{2n0}^{\prime}(z_{2})}{n}\sum_{i=1}^{p}\left[ \bGma^{*}\bbP_{n}^{2}(z_{1})\bGma\right] _{ii}\left[ \bGma^{*}\bbP_{n}^{2}(z_{2})\bGma\right] _{ii},
\end{align*}
and the definitions of $ \bbP_{n} $, $ \bGma $, and $ \underline{m}_{2n0} $ are defined in Section \ref{section 3}.
\end{lemma}
\begin{proof}
	From \cite{10.1214/14-AOS1292}, we have that
under Assumptions \ref{ass1}- \ref{ass4}, the random variable  $ \left( \kappa_{1}^{0}\right) ^{-1}\left( Q_{1}-\mu_{1}\right) \stackrel{d}\longrightarrow N\left(0,1 \right), $ with mean function
\begin{align*}
	\mu_{1}&=-\frac{\alpha_{x}}{2 \pi i}\cdot\oint_{\mathcal{C}}\frac{ f_{1}(z) c_{nM} \int \underline{m}_{2n0}^{3}(z)t^{2}\left(1+t \underline{m}_{2n0}(z)\right)^{-3} d H_{2n}(t)}{\left(1-c_{nM} \int \frac{\underline{m}_{2n0}^{2}(z) t^{2}}{\left(1+t \underline{m}_{2n0}(z)\right)^{2}} d H_{2n}(t)\right)\left(1-\alpha_{x} c_{nM} \int \frac{\underline{m}_{2n0}^{2}(z) t^{2}}{\left(1+t \underline{m}_{2n0}(z)\right)^{2}} d H_{2n}(t)\right) }dz \\
	&-\frac{\beta_{x}}{2 \pi i} \cdot \oint_{\mathcal{C}} \frac{f_{1}(z) c_{nM} \int \underline{m}_{2n0}^{3}(z) t^{2}\left(1+t \underline{m}_{2n0}(z)\right)^{-3} d H_{2n}(t)}{1-c_{nM} \int \underline{m}_{2n0}^{2}(z) t^{2}\left(1+t \underline{m}_{2n0}(z)\right)^{-2} d H_{2n}(t)} dz,  	
\end{align*}
and the covariance function is  \begin{align*}
			\left( \kappa_{1}^{0}\right)^{2}=&-\frac{1}{4\pi^{2}}\oint_{\mathcal{C}_{1}}\oint_{\mathcal{C}_{2}}f_{1}\left(z_{1} \right)f_{1}\left(z_{2} \right)(\vartheta_{n}^{0})^{2}dz_{1}dz_{2}, 
		\end{align*} 
		where
		\begin{align*}
			(\vartheta_{n}^{0})^{2}=&\frac{b_{n}\left(z_{1}\right) b_{n}\left(z_{2}\right)}{n^{2}} \sum_{j=1}^{n} \operatorname{tr} \mathbb{E}_{j} \bGma\bGma^{*} \mathbf{A}_{j}^{-1}\left(z_{1}\right) \mathbb{E}_{j}\left(\bGma\bGma^{*} \mathbf{A}_{j}^{-1}\left(z_{2}\right)\right)\\&+\frac{\alpha_{x} b_{n}\left(z_{1}\right) b_{n}\left(z_{2}\right)}{n^{2}} \sum_{j=1}^{n} \operatorname{tr} \mathbb{E}_{j} \bGma^{*} \mathbf{A}_{j}^{-1}\left(z_{1}\right)\bGma
			\mathbb{E}_{j}\left(\bGma^{'}\left(\mathbf{A}_{j}^{\prime}\right)^{-1}\left(z_{2}\right)\bar{\bGma}\right)\\
			&+
			\frac{\beta_{x} b_{n}\left(z_{1}\right) b_{n}\left(z_{2}\right)}{n^{2}} \sum_{j=1}^{n} \sum_{i=1}^{p} \boldsymbol{e}_{i}^{\prime}  {\bGma}^{*} \mathbf{A}_{j}^{-1}\left(z_{1}\right) {\bGma} \boldsymbol{e}_{i} \cdot \boldsymbol{e}_{i}^{\prime} {\bGma}^{*} \mathbf{A}_{j}^{-1}\left(z_{2}\right) {\bGma} \boldsymbol{e}_{i},
		\end{align*}
		where $ b_{n}\left(z \right)=\frac{1}{1+n^{-1}\mathbb{E}\rtr\bGma\bGma^{*}\bbA_{j}^{-1}\left(z \right) }.$ The notation $ \bbA_{j}, \boldsymbol{e_{i}} $ is defined in Section \ref{section 2}. 
Moreover \cite{najim2016gaussian} provided an estimation $ \vartheta_{n}^{2} $ for $ (\vartheta_{n}^{0})^{2} $ and proved that $ (\vartheta_{n}^{0})^{2} $ is close to $ \vartheta_{n}^{2} $ in the Lévy--Prohorov distance, where 
$\vartheta_{n}^{2}=\Theta_{0,n}(z_{1},z_{2})+\al_{x}\Theta_{1,n}(z_{1},z_{2})+\beta_{x}\Theta_{2,n}(z_{1},z_{2}),$ 
\begin{align*}
	\Theta_{0,n}(z_{1},z_{2})&=\dfrac{\underline{m}_{2n0}^{\prime}(z_{1}) \underline{m}_{2n0}^{\prime}(z_{2})  }{(\underline{m}_{2n0}(z_{1})-\underline{m}_{2n0}(z_{2}) )^{2} }-\dfrac{1}{(z_{1}-z_{2})^{2}},\\
	\Theta_{1,n}(z_{1},z_{2})&=\frac{\partial}{\partial z_{2}}\left\lbrace \dfrac{\partial \mathcal{A}_{n}(z_{1},z_{2})}{\partial z_{1}}\dfrac{1}{1-\al_{x}\mathcal{A}_{n}(z_{1},z_{2})} \right\rbrace ,\\
	\mathcal{A}_{n}(z_{1},z_{2})&=\dfrac{z_{1}z_{2}}{n}\underline{m}_{2n0}^{\prime}(z_{1}) \underline{m}_{2n0}^{\prime}(z_{2})\mathrm{tr}{\bGma^{*}\bbP_{n}(z_{1})\bGma\bGma^{\prime}\bbP_{n}(z_{2})^{\prime} \bar{\bGma}},\\
	\Theta_{2,n}(z_{1},z_{2})&=\dfrac{z_{1}^{2}z_{2}^{2}\underline{m}_{2n0}^{\prime}(z_{1}) \underline{m}_{2n0}^{\prime}(z_{2})}{n}\sum_{i=1}^{p}\left[ \bGma^{*}\bbP_{n}^{2}(z_{1})\bGma\right] _{ii}\left[ \bGma^{*}\bbP_{n}^{2}(z_{2})\bGma\right] _{ii},
\end{align*}
The definitions of $ \bbP_{n} $, $ \bGma $, and $ \underline{m}_{2n0} $ are defined in Section \ref{section 3}. Notably, if $ \bGma $ is not real, 
the convergence of $ \Theta_{1,n}(z_{1},z_{2}) $ is not granted. However, if $ \bGma $ and entries $ x_{ij} $ are real, that is, $ \al_{x}=1 $, then it can be easily proven that $ \Theta_{0,n}(z_{1},z_{2})= \Theta_{1,n}(z_{1},z_{2}) $. Similarly, the convergence of $ \Theta_{2,n}(z_{1},z_{2}) $ depends on the assumption that $ \bGma^{*}\bGma $ is diagonal; thus, under Assumptions \ref{ass1}-\ref{ass4}, $ \Theta_{1,n}(z_{1},z_{2}) $ and $ \Theta_{2,n}(z_{1},z_{2}) $ may have no limits.

Thus, the covariance term $ \left( \kappa_{1}^{0}\right) ^{2} $ is estimable, and the estimation is $ \kappa_{1}^{2} $, with $$  \kappa_{1}^{2}=  -\frac{1}{4\pi^{2}}\oint_{\mathcal{C}_{1}}\oint_{\mathcal{C}_{2}}f_{1}\left(z_{1} \right)f_{1}\left(z_{2} \right)\vartheta_{n}^{2}dz_{1}dz_{2}. $$
 Therefore, the proof is finished.
\end{proof}

 \subsection{Proof of Theorem \ref{thm1}}



The proof of Theorem \ref{thm1} builds on the decomposition analysis of the LSSs and is divided into part (\uppercase\expandafter{\romannumeral1}) $ \sum_{j=1}^{M}f\left( \bgl_{j}\right) $ and part (\uppercase\expandafter{\romannumeral2}) $ \sum_{j=M+1}^{p}f\left( \bgl_{j}\right) $. Enlightened by the BST in \cite{10.1214/aop/1078415845}, we have
\begin{align*}
	&\sum_{j=1}^{p}f\left( \bgl_{j}\right)-p\int f(x)dF^{c_{n},H_{n}} \\
	&=\sum_{j=1}^{M}f\left( \bgl_{j}\right)+\sum_{j=M+1}^{p}f\left( \bgl_{j}\right)-p\int f(x)dF^{c_{n},H_{n}}\\
	&=\sum_{j=1}^{M}f\left( \bgl_{j}\right)+\sum_{j=1}^{p-M}f\left( \widetilde{\bgl}_{j}\right)-(p-M)\int f(x)dF^{c_{nM},H_{2n}}+\sum_{j=M+1}^{p}f\left( \bgl_{j}\right)-\sum_{j=1}^{p-M}f\left( \widetilde{\bgl}_{j}\right)\\
	&+(p-M)\int f(x)dF^{c_{nM},H_{2n}}-p\int f(x)dF^{c_{n},H_{n}}.
\end{align*}
Since Lemma \ref{lemma1} has shown the difference between $ (p-M)\int f(x)dF^{c_{nM},H_{2n}}$ and $p\int f(x)dF^{c_{n},H_{n}}$ is 0. 
Moreover, in Lemma \ref{lemma2} we have proved $$ \sum_{j=M+1}^{p}f\left( \bgl_{j}\right)-\sum_{j=1}^{p-M}f\left( \widetilde{\bgl}_{j}\right)=\frac{M}{2\pi i}\oint_{C}f\left(z \right)\frac{\underline{m}^{'}(z)}{\underline{m}(z)}dz+o_{p}(1). $$
It follows that
\begin{align*}
	&\sum_{j=1}^{p}f\left( \bgl_{j}\right)-p\int f(x)dF^{c_{n},H_{n}}\\
	=&\sum_{j=1}^{M}f\left( \bgl_{j}\right)+\sum_{j=1}^{p-M}f\left( \widetilde{\bgl}_{j}\right)
	-(p-M)\int f(x)dF^{c_{nM},H_{2n}}+\frac{M}{2\pi i}\oint_{C}f\left(z \right)\frac{\underline{m}^{'}(z)}{\underline{m}(z)}dz+ o_{p}(1),
\end{align*}
which yields
\begin{align}
	&\sum_{j=1}^{p}f\left( \bgl_{j}\right)-p\int f(x)dF^{c_{n},H_{n}}-\sum_{k=1}^{K}d_{k}f\left( \phi_{n}\left(\al_{k} \right) \right)-\frac{M}{2\pi i}\oint_{C}f\left(z \right)\frac{\underline{m}^{'}(z)}{\underline{m}(z)}dz\label{101}\\
	=&\sum_{j=1}^{M}f\left( \bgl_{j}\right)-\sum_{k=1}^{K}d_{k}f\left( \phi_{n}\left(\al_{k} \right) \right)+\sum_{j=1}^{p-M}f\left( \widetilde{\bgl}_{j}\right)-(p-M)\int f(x)dF^{c_{nM},H_{2n}}+ o_{p}(1). \label{100}
\end{align}

The analysis below is executed by dividing (\ref{100}) into two parts: (\uppercase\expandafter{\romannumeral1}) $ \sum_{j=1}^{M}f\left( \bgl_{j}\right)-\sum_{k=1}^{K}d_{k}f\left( \phi_{n}\left(\al_{k} \right) \right) $ and (\uppercase\expandafter{\romannumeral2}) $ \sum_{j=1}^{p-M}f\left( \widetilde{\bgl}_{j}\right)-(p-M)\int f(x)dF^{c_{nM},H_{2n}} $, where we ignore the impact of $ o_{p}(1) $ on the asymptotic distribution. Since we have derived the asymptotic distribution of part (\uppercase\expandafter{\romannumeral2}) in Lemma \ref{lemma3}, we only need to consider the asymptotic distribution of part (\uppercase\expandafter{\romannumeral1}) $ \sum_{j=1}^{M}f\left(\lambda_{j}\right)-\sum_{k=1}^{K}d_{k}f\left(\phi_{n}\left(\bbalp_{k} \right) \right) $. From the proof of Lemma \ref{lemma5}, $ \sum_{j=1}^{M}f\left( \bgl_{j}\right)-\sum_{k=1}^{K}d_{k}f\left(\phi_{n}\left(\bbalp_{k} \right)\right)  $ has the same limiting distribution as $ \sum_{k=1}^{K}\frac{\phi_{n}\left(\bbalp_{k} \right)}{\sqrt{n}}f'\left( \phi_{n}\left(\bbalp_{k} \right)\right)\sum_{j\in J_{k}}\gamma_{kj}  $. From Lemma \ref{lemma4}, we have $ \left(\gamma_{kj}, j\in J_{k} \right)'\stackrel{d}\rightarrow-\frac{1}{\theta_{k}}\left[\bgO_{\phi_{k}} \right]_{kk}  $, so $ \sum_{j\in J_{k}}\gamma_{kj}\stackrel{d}\rightarrow -\frac{1}{\theta_{k}}\mathrm {tr}\left[\bgO_{\phi_{k}} \right]_{kk}    $. Recall that $ \omega_{ij} $ is the element of $ \bgO_{\phi_{k}} $, and $ \mathrm {tr}\left[\bgO_{\phi_{k}} \right]_{kk} $ is the summation of the diagonal element, that is, $ \sum_{j\in J_{k}} \omega_{jj} $. Because the diagonal elements are i.i.d.,  $ \mathbb{E}\left(\sum_{j\in J_{k}} \omega_{jj}\right) =0  $, $ \mathrm{Var}\left(\sum_{j\in J_{k}} \omega_{jj}\right)=\sum_{j\in J_{k}}\mathrm{Var}\left(\omega_{jj} \right)+ \sum_{j_{1}\neq j_{2}}\mathrm{cov}\left(\omega_{j_{1}j_{1}}, \omega_{j_{2}j_{2}} \right)= \sum_{j\in J_{k}}\left( \left(\al_{x}+1 \right)\theta_{k}+\beta_{x}\mathcal{U}_{jjjj}\nu_{k}\right) +\sum_{j_{1}\neq j_{2}}\beta_{x}\mathcal{U}_{j_{1}j_{1}j_{2}j_{2}}\nu_{k}.  $ 

Therefore,
from Lemma \ref{lemma4}, we have that the asymptotic distribution of $ \sum_{j\in J_{k}}\gamma_{kj} $ is a Gaussian distribution with $$
\mathbb{E}\sum_{j\in J_{k}}\gamma_{kj}=0,$$ $$
s_{k}^{2}\triangleq \mathrm{Var}\left(\sum_{j\in J_{k}}\gamma_{kj}\right)=\frac{ \sum_{j\in J_{k}}\left( \left(\al_{x}+1 \right)\theta_{k}+\beta_{x}\mathcal{U}_{jjjj}\nu_{k}\right) +\sum_{j_{1}\neq j_{2}}\beta_{x}\mathcal{U}_{j_{1}j_{1}, j_{2}j_{2}}\nu_{k}}{\theta_{k}^{2}}, $$ and then, we directly derive that the mean function of $ \sum_{k=1}^{K}\frac{\phi_{n}\left(\bbalp_{k} \right)}{\sqrt{n}}f'\left( \phi_{n}\left(\bbalp_{k} \right)\right)\sum_{j\in J_{k}}\gamma_{kj}  $ is 0 and that its covariance function is
$$
\mathrm{Var}\left( Y_{f_{1}} \right)=\sum_{k=1}^{K}\frac{\phi_{n}^{2}\left(\bbalp_{k} \right)}{n}\left( f_{1}'\left( \phi_{n}\left(\bbalp_{k} \right)\right)\right) ^{2}s_{k}^{2}. 
$$

Finally, we focus on the asymptotic distribution of equation (\ref{100}). Because of Lemma \ref{lemma5}, the two LSSs are asymptotically independent; thus, the random variable 
$$ \varsigma_{1}^{-1}\left(Y_{1}-\mathbb{E}Y_{1} \right)\stackrel{d}\longrightarrow N\left(0,1 \right)  $$
with mean function $\mathbb{E}Y_{1}=\mu_{1}$ being
\begin{align*}
	&-\frac{\alpha_{x}}{2 \pi i}\cdot\oint_{\mathcal{C}}f_{1}(z)\frac{  c_{nM} \int \underline{m}_{2n0}^{3}(z)t^{2}\left(1+t \underline{m}_{2n0}(z)\right)^{-3} d H_{2n}(t)}{\left(1-c_{nM} \int \frac{\underline{m}_{2n0}^{2}(z) t^{2}}{\left(1+t \underline{m}_{2n0}(z)\right)^{2}} d H_{2n}(t)\right)\left(1-\alpha_{x} c_{nM} \int \frac{\underline{m}_{2n0}^{2}(z) t^{2}}{\left(1+t \underline{m}_{2n0}(z)\right)^{2}} d H_{2n}(t)\right) }dz \\
	&-\frac{\beta_{x}}{2 \pi i} \cdot \oint_{\mathcal{C}} f_{1}(z) \frac{c_{nM} \int \underline{m}_{2n0}^{3}(z) t^{2}\left(1+t \underline{m}_{2n0}(z)\right)^{-3} d H_{2n}(t)}{1-c_{nM} \int \underline{m}_{2n0}^{2}(z) t^{2}\left(1+t \underline{m}_{2n0}(z)\right)^{-2} d H_{2n}(t)} dz,  	
\end{align*}
and covariance function $\varsigma_{1}^{2}$ being 
\begin{align*}
	\sum_{k=1}^{K}\frac{\phi_{n}^{2}\left(\al_{k} \right)}{n}\left( f_{1}'\left(\phi_{n}\left(\al_{k} \right)\right)  \right)^{2}s_{k}^{2} -\frac{1}{4\pi^{2}}\oint_{\mathcal{C}_{1}}\oint_{\mathcal{C}_{2}}f_{1}\left(z_{1} \right)f_{1}\left(z_{2} \right)(\vartheta_{n})^{2}dz_{1}dz_{2},
\end{align*}
where $ \vartheta_{n}^{2} $ is defined in Lemma \ref{lemma3}. Therefore, the proof is finished.

\subsection{Proof of Theorem \ref{thm2}}
Similar to the proof of Theorem \ref{thm1}, we divide the LSSs into two parts. Different from the above analysis, in this section, we focus on the multidimensional case under Assumptions \ref{ass1} -\ref{ass6}. Recall that we defined
$$ G_{n}\left( x\right)=p\left[F^{\bbB}\left(x \right)-F^{c_{n},H_{n}}\left(x \right)   \right],  $$ 
$$ Y_{l}=  \int f_{l}\left(x \right)dG_{n}\left( x\right)-\sum_{k=1}^{K}d_{k}\left(\phi_{n}\left(\al_{k} \right)  \right)-\frac{M}{2\pi i}\oint_{\mathcal C}f_{l}\left(z \right)\frac{\underline{m}'(z)}{\underline{m}(z)}dz. $$
Because of equation (\ref{100}), the random vector $\left(  Y_{1},\dots, Y_{h} \right)  $ shares the same asymptotic distribution with the summation of two random vectors  $$ \left( \sum_{k=1}^{K}\frac{\phi_{n}\left(\bbalp_{k} \right)}{\sqrt{n}}f_{1}'\left( \phi_{n}\left(\bbalp_{k} \right)\right)\sum_{j\in J_{k}}\gamma_{kj} ,\dots,\sum_{k=1}^{K}\frac{\phi_{n}\left(\bbalp_{k} \right)}{\sqrt{n}}f_{h}'\left( \phi_{n}\left(\bbalp_{k} \right)\right)\sum_{j\in J_{k}}\gamma_{kj}  \right)$$ and  $$\left(   \sum_{j=1}^{p-M}f_{1}\left( \widetilde{\bgl}_{j}\right)-(p-M)\int f_{1}(x)dF^{c_{nM},H_{2n}} ,\dots, \sum_{j=1}^{p-M}f_{h}\left( \widetilde{\bgl}_{j}\right)-(p-M)\int f_{h}(x)dF^{c_{nM},H_{2n}} \right).    $$ 

First, we focus on the first random vector. Similar to the proof of Theorem \ref{thm1}, we derive that the mean function of the first random vector is 0 and that the covariance function is
$$
\mathrm{Cov}\left( Y_{f_{s}},Y_{f_{t}} \right)=\sum_{k=1}^{K}\frac{\phi_{n}^{2}\left(\bbalp_{k} \right)}{n} f_{s}'\left( \phi_{n}\left(\bbalp_{k} \right)\right)f_{t}'\left( \phi_{n}\left(\bbalp_{k} \right)\right)s_{k}^{2}\rightarrow \sum_{k=1}^{K}\varpi_{s}^{k}\varpi_{t}^{k}s_{k}^{2}, \quad n\rightarrow \infty,
$$
where $  \dfrac{\phi_{n}(\al_{k})}{\sqrt{n}}f_{l}^{\prime}(\phi_{n}(\al_{k}) )\rightarrow\varpi_{l}^{k}$ as $ n\rightarrow\infty $. 
Moreover, the asymptotic distribution of the second random vector is derived in \cite{10.1214/14-AOS1292}. Because of Lemma \ref{lemma5}, two random vectors are asymptotically independent; thus, the random vector
$$ \left( Y_{1}-\mathbb{E}Y_{1},\dots,Y_{h}-\mathbb{E}Y_{h}\right)'\stackrel{d}{\rightarrow}N_{h}\left(0, \bgO\right),   $$ 
with mean function $ \mathbb{E}Y_{l} $ is the same as $ \mu_{l} $, and the covariance matrix is $ \bgO $ with its entries
\begin{align*}
	\omega_{st}=\sum_{k=1}^{K}\varpi_{s}^{k}\varpi_{t}^{k}s_{k}^{2}+ \kappa_{st},
\end{align*} 
where \begin{align*}
	\kappa_{st}
	&=-\frac{1}{4 \pi^{2}} \oint_{\mathcal{C}_{1}} \oint_{\mathcal{C}_{2}} \frac{f_{s}\left(z_{1}\right) f_{t}\left(z_{2}\right)}{\left(\underline{m} {\left.\left(z_{1}\right)-\underline{m}\left(z_{2}\right)\right)^{2}}\right.} d \underline{m}\left(z_{1}\right) d \underline{m}\left(z_{2}\right) 
	-\frac{c \beta_{x}}{4 \pi^{2}} \oint_{\mathcal{C}_{1}} \oint_{\mathcal{C}_{2}} f_{s}\left(z_{1}\right) f_{t}\left(z_{2}\right)\\
	&\left[\int \frac{t}{\left(\underline{m}\left(z_{1}\right) t+1\right)^{2}}\right.
	\left.\times \frac{t}{\left(\underline{m}\left(z_{2}\right) t+1\right)^{2}} d H(t)\right] d \underline{m}\left(z_{1}\right) d \underline{m}\left(z_{2}\right)\\
	&\quad-\frac{1}{4 \pi^{2}} \oint_{\mathcal{C}_{1}} \oint_{\mathcal{C}_{2}} f_{s}\left(z_{1}\right) f_{t}\left(z_{2}\right)\left[\frac{\partial^{2}}{\partial z_{1} \partial z_{2}} \log \left(1-a\left(z_{1}, z_{2}\right)\right)\right] d z_{1} d z_{2},\\
	a\left(z_{1}, z_{2}\right)&=\alpha_{x}\left(1+\frac{\underline{m}\left(z_{1}\right) \underline{m}\left(z_{2}\right)\left(z_{1}-z_{2}\right)}{\underline{m}\left(z_{2}\right)-\underline{m}\left(z_{1}\right)}\right).
\end{align*} 
Then, we obtain the random vector $$ \left( \frac{Y_{1}-\mathbb{E}Y_{1}}{\sqrt{\sigma_{1}^{2}}},\dots,\frac{Y_{h}-\mathbb{E}Y_{h}}{\sqrt{\sigma_{h}^{2}}}\right)'\stackrel{d}{\rightarrow}N_{h}\left(0, \bPsi \right), 
$$ which has a mean function that is the same as that in Theorem \ref{thm1}, and variance function 
\begin{align*}
	\sigma_{l}^{2}&=\sum_{k=1}^{K}\frac{\phi_{n}^{2}(\al_{k})}{n}(f_{l}^{\prime }(\phi_{n}(\al_{k})))^{2} s_{k}^{2}+\kappa_{ll}, \quad l=1,\dots,h,
\end{align*} 
and the covariance matrix $  \bPsi =\left( \psi_{st}\right) _{h\times h} $ is the correlation coefficient matrix of random vector $ (Y_{1},\dots,Y_{h})^{\prime} $  with its entries
\begin{align*}
	\psi_{st}=\frac{\sum_{k=1}^{K}\varpi_{s}^{k}\varpi_{t}^{k}s_{k}^{2}+ \kappa_{st}}{\sqrt{\sum_{k=1}^{K}(\varpi_{s}^{k})^{2}s_{k}^{2}+ \kappa_{ss}}\sqrt{\sum_{k=1}^{K}(\varpi_{t}^{k})^{2}s_{k}^{2}+ \kappa_{tt}}},
\end{align*} 
Note that renormalization is necessary to guarantee that elements in the correlation coefficient matrix $ \bPsi $ are limited. Therefore, the proof is finished.

\subsection{Proof of Theorem \ref{thm3}}

The result under $ H_{0} $ is a direct result of applying Theorem 4.1 in \cite{10.1214/14-AOS1292}  using the substitution principle. Therefore, we omit the proof here. Next, we focus on the result under $ H_{1}. $

Recall that $$ G_{n}\left( x\right)=p\left[F^{\bbB}\left(x \right)-F^{c_{n},H_{n}}\left(x \right)   \right],  $$ $$  Y= \int f_{L}\left(x \right)dG_{n}\left( x\right)-\sum_{k=1}^{K}d_{k}f_{L}\left(\phi_{n}\left(\al_{k} \right)  \right)-\frac{M}{2\pi i}\oint_{\mathcal C}f_{L}\left(z \right)\frac{\underline{m}'(z)}{\underline{m}(z)}dz,$$
when $ f_{L}(x)=x-\log x-1. $ After some calculations, we obtain
\begin{gather}
	\nonumber\int f_{L}\left(x \right)dG_{n}\left( x\right)=\mathrm{tr}\bbB-\log \left|\bbB \right|-p-p\int f_{L}(x)dF^{c_{n},H_{n}}(x)=L-p\int f_{L}(x)dF^{c_{n},H_{n}}(x),\\
	p\int f_{L}(x)dF^{c_{n},H_{n}}(x)=(p-M)(1-\frac{c_{nM}-1}{c_{nM}} \log \left(1-c_{nM}\right)),\label{32}\\	
	\nonumber\sum_{k=1}^{K}d_{k}f_{L}\left(\phi_{n}\left(\al_{k} \right)  \right)=\sum_{k=1}^{K}d_{k}\left(\phi_{n}\left(\al_{k} \right)-\log \phi_{n}\left(\al_{k} \right)-1 \right),\\
	\frac{M}{2\pi i}\oint_{\mathcal C}f_{L}\left(z \right)\frac{\underline{m}'(z)}{\underline{m}(z)}dz= -M(c_{n}+\log(1-c_{n})),\label{31},
\end{gather} 
where (\ref{32}) is obtained from Lemma \ref{lemma1} and \cite{10.1214/09-AOS694}. For consistency, we present the proof of (\ref{31}) in Appendix \ref{appB}.
According to Theorem \ref{thm1}, when $ f_{L}(x)=x-\log x-1 $, we have
\begin{align*}
 	 \frac{L-p\int f_{L}(x)dF^{c_{n},H_{n}}(x)-\breve{\mu}_{L}}{ \sqrt{\breve{\varsigma}_{L}^{2}}}\stackrel{d}{\rightarrow}N(0,1),
\end{align*}	where the mean function is $\breve{\mu}_{L}=-\frac{\log \left(1-c_{nM}\right)}{2}\al_{x}+\frac{c_{nM}}{2} \beta_{x} 
 +\sum_{k=1}^{K}d_{k}\left( \phi_{n}\left({\al}_{k} \right)-\log\phi_{n}\left({\al}_{k} \right)-1 \right)-M(c_{n}+\log(1-c_{n}))$, and the covariance function $ \breve{\varsigma}_{L}^{2} $ is $\frac{\al_{x}+1}{2} (-2 \log \left(1-c_{nM}\right)-2 c_{nM})+\sum_{k=1}^{K}d_{k}\frac{2\left( \phi_{n}\left(\al_{k} \right)-1\right)^{4}}{n\phi_{n}^{2}\left(\al_{k} \right) } $. 

 Now, we consider the power of the hypothesis. Let $ a $ be the size of the hypothesis in Section \ref{section 5}; $ z_{a} $ is the upper $ a\% $ quantile of the standard Gaussian distribution $ \Phi$. Since
$$ P_{H_{0}}\left(\left| L\right| >w \right)=a=P_{H_{0}}\left(\frac{\left| L\right| -L_{0}}{\varsigma_{L}} >\frac{w-L_{0}}{\varsigma_{L}}\right),   $$ 
here, $ w $ is a threshold of the critical region $ L_{0}=p\int f_{L}(x)dF^{c_{n},H_{n}}+\mu_{L} $. Then, we have $ w=\varsigma_{L}z_{a}+L_{0}. $ For brevity, we use the notation $ L_{1}=(p-M)\int f_{L}(x)dF^{c_{nM},H_{2n}}   +\breve{\mu}_{L}. $ Therefore, the power of the hypothesis is
\begin{align*}
	P_{H_{1}}\left(\left| L\right| >w \right)&=P_{H_{1}}\left(\frac{\left| L\right| -L_{1}}{\breve{\varsigma}_{L}} >\frac{w-L_{1}}{\breve{\varsigma}_{L}}\right)\\
	&=\Phi\left(\frac{L_{1}-L_{0}-\varsigma_{L}z_{a}}{\breve{\varsigma}_{L}} \right)= \Phi\left( \frac{L_{1}-L_{0}}{\breve{\varsigma}_{L}}-z_{a}\frac{\varsigma_{L}}{\breve{\varsigma}_{L}}\right).
\end{align*}
When $ M=1 $, 
\begin{align*} L_{1}-L_{0}&=(p-1)(1-\frac{c_{n1}-1}{c_{n1}}\log(1-c_{n1}))-p(1-\frac{c_{n}-1}{c_{n}}\log(1-c_{n}))+\\&(\phi_{n}(\al_{1})-\log \phi_{n}(\al_{1})-1)-(c_{n}+\log(1-c_{n})  )-\frac{\log(1-c_{n1})}{2}\al_{x}+\frac{\log(1-c_{n})}{2}\al_{x}\\&+\frac{c_{n1}}{2}\beta_{x}-\frac{c_{n}}{2}\beta_{x},  
\end{align*}
after some elementary calculations, $ \phi(\al_{1})=\al_{1}+\frac{c\al_{1}}{\al_{1}-1}, $  we obtain as $ n\rightarrow \infty, $ $$ L_{1}-L_{0}\rightarrow-c+(\phi_{1}-\log \phi_{1}-1 ), $$
$$ \varsigma_{L}\rightarrow\sqrt{\frac{\al_{x}+1}{2} (-2\log(1-c)-2c)},   $$
$$ \breve{\varsigma}_{L}\rightarrow\sqrt{\frac{\al_{x}+1}{2} (-2\log(1-c)-2c)  +\frac{2}{n}\frac{(\phi_{1}-1)^{4}}{\phi_{1}^{2}}}.     $$
Therefore, the asymptotic power of LRT is $$  \Phi( \frac{ -c+  \left( \phi_{1}-\log\phi_{1}-1 \right)-z_{a}\sqrt{\frac{\al_{x}+1}{2} (-2\log(1-c)-2c) }  }{ \sqrt{\frac{\al_{x}+1}{2} (-2\log(1-c)-2c)  +\frac{2}{n}\frac{(\phi_{1}-1)^{4}}{\phi_{1}^{2}}}   }   )  .$$
Thus, the proof of Theorem \ref{thm3} is finished. 
\subsection{Proof of Theorem \ref{thm4}}
First, we focus on the results under $ H_{0} $. From Remark \ref{simple result}, we have
\begin{align}
I_{1}(f_{W})&=c,\label{I1}	\\
I_{2}(f_{W})&=c, \label{I2}\\
J_{1}(f_{W},f_{W})&=4c^{3}+2c^{2},\label{J1}\\
J_{2}(f_{W},f_{W})&=4c^{3},\label{J2},
\end{align}	
which then yields
\begin{align*}
	\mu_{W}&=\al_{x}I_{1}(f_{W})+\beta_{x}I_{2}(f_{W})=\al_{x}c+\beta_{x}c,\\
	\varsigma_{W}^{2}&=(\al_{x}+1)J_{1}(f_{W},f_{W})+\beta_{x}J_{2}(f_{W},f_{W})= (\al_{x}+1)(4c^{3}+2c^{2})+4\beta_{x}c^{3}.
\end{align*}
The results are still valid if $ c $ is replaced by $ c_{n} $. Moreover, the center term
 \begin{align}\label{center term}
	\int f_{W}(x) dF^{c_{n},H_{n}} = c_{n},	
\end{align}	
is a direct result of Lemma 2.2 in \cite{wang2013sphericity}.
The proofs of (\ref{I1}), (\ref{I2}), (\ref{J1}) and (\ref{J2}) are presented in Appendix \ref{appB}. Therefore, from \cite{10.1214/14-AOS1292} or \cite{wang2013sphericity}, we have
\begin{align*}
	\frac{W-p\int f_{W}(x)dF^{c_{n},H_{n}}-\mu_{W}  }{\sqrt{\varsigma_{W}^{2}}}\stackrel{d}{\longrightarrow} N(0,1).
\end{align*}
Then, we focus on the results under $ H_{1} $. Note that
\begin{align*}
Y= \int f_{W}\left(x \right)dG_{n}\left( x\right)-\sum_{k=1}^{K}d_{k}f_{W}\left(\phi_{n}\left(\al_{k} \right)  \right)-\frac{M}{2\pi i}\oint_{\mathcal C}f_{W}\left(z \right)\frac{\underline{m}'(z)}{\underline{m}(z)}dz.
\end{align*}	
After some calculations, we obtain
\begin{gather}
	\nonumber\int f_{W}\left(x \right)dG_{n}\left( x\right)=\mathrm{tr}(\bbB-\bbI_{p})^{2}-p\int f_{W}(x)dF^{c_{n},H_{n}}=W-p\int f_{W}(x)dF^{c_{n},H_{n}},\\
	p\int f_{W}(x)dF^{c_{n},H_{n}}=(p-M)\int f_{W}(x)dF^{c_{nM},H_{2n}}=(p-M)c_{nM}\\
	\nonumber\sum_{k=1}^{K}d_{k}f_{W}\left(\phi_{n}\left(\al_{k} \right)  \right)=\sum_{k=1}^{K}d_{k}\left(\phi_{n}^{2}\left(\al_{k} \right)-2\phi_{n}\left(\al_{k} \right)+1 \right),\\
	\frac{M}{2\pi i}\oint_{\mathcal C}f_{W}\left(z \right)\frac{\underline{m}'(z)}{\underline{m}(z)}dz= -M c_{n}^{2}
	\label{l_{n}}
\end{gather} 
For consistency, we present the proof of (\ref{l_{n}}) in Appendix \ref{appB}. Therefore, from Theorem \ref{thm1}, we have
\begin{align*}
	\dfrac{W-(p-M)\breve{\ell}_{W}-\breve{\mu}_{W}}{\sqrt{\breve{\varsigma}^{2}_{W}}}\stackrel{d}{\longrightarrow} N(0,1),
\end{align*}	
where 
\begin{align*}
	\breve{\ell}_{W}&=c_{nM},~~~
	\breve{\mu}_{W}=\al_{x}c_{nM}+\beta_{x}c_{nM}+\sum_{k=1}^{K}d_{k}\left( \phi_{n}^{2}\left({\al}_{k} \right)-2\phi_{n}\left({\al}_{k} \right)+1 \right)-M c_{n}^{2}, \\
\breve{\varsigma}^{2}_{W}&= (\al_{x}+1)(4c_{nM}^{3}+2c_{nM}^{2})+4\beta_{x}c_{nM}^{3}+\sum_{k=1}^{K}d_{k}\frac{8\left( \phi_{n}\left(\al_{k} \right)-1\right)^{4}}{n}.
\end{align*}
Moreover, the power analysis for NT is similar to that for the LR; thus, we omit the detailed proof here.
Therefore, the proof of Theorem \ref{thm4} is complete.

\begin{appendix}

\section{Some derivations and calculations}\label{appB}

This section contains proof of formulas stated in the proof of Theorems \ref{thm3} and \ref{thm4}, and we begin by deriving formula (\ref{31}).
 First, we consider $ \oint_{\mathcal C}f_{L}\left(z \right)\frac{\underline{m}'(z)}{\underline{m}(z)}dz $. 
\begin{align}
	&\nonumber\oint_{\mathcal C}f_{L}\left(z \right)\frac{\underline{m}'(z)}{\underline{m}(z)}dz
	=\nonumber\oint_{\mathcal C}f_{L}\left(z \right)d\log\underline{m}\left(z \right)=\nonumber-\oint_{\mathcal C}f_{L}^{'}\left(z \right)\log\underline{m}\left(z \right)dz \\
	=&\nonumber\int_{a(c)}^{b(c)}f_{L}^{'}\left(z \right)\left[\log\underline{m}(x+i\varepsilon)-\log\underline{m}(x-i\varepsilon) \right]dx\\
	=&2i 	\int_{a(c)}^{b(c)}f_{L}^{'}\left(z \right) \Im \log\underline{m}(x+i\varepsilon)dx \label{26}
\end{align}	
Here, $a(c)=(1-\sqrt{c})^{2}  $ and $b(c)=(1+\sqrt{c})^{2}  $. Since $$\underline{m}\left(z \right)=-\frac{1-c}{z}+cm\left(z \right),  $$
under $ H_{1} $, we have $$ \underline{m}\left(z \right)=\frac{-(z+1-c)+\sqrt{(z-1-c)^{2}-4c}}{2z} . $$
As $z \rightarrow x \in$ $[a(c), b(c)]$, we obtain
$$ \underline{m}\left(x \right)= \frac{-(x+1-c)+\sqrt{4c-(x-1-c)^{2}}i}{2x}.$$
Therefore,
\begin{align*}
	&\int_{a(c)}^{b(c)}f_{L}^{'}\left(z \right) \Im \log\underline{m}(x+i\varepsilon)dx\\
	&=\int_{a(c)}^{b(c)} f_{L}^{'}(x)\tan ^{-1}\left(\frac{\sqrt{4 c-(x-1-c)^{2}}}{-(x+1-c)}\right) d x\\
	&=\left[\left.\tan ^{-1}\left(  \frac{\sqrt{4 c-(x-1-c)^{2}}}{-(x+1-c)}\right)  f_{L}(x)\right|_{a(c)} ^{b(c)}-\int_{a(c)}^{b(c)} f_{L}(x) d \tan ^{-1}\left(\frac{\sqrt{4 c-(x-1-c)^{2}}}{-(x+1-c)}\right)\right].
\end{align*}
It is easy to verify that the first term is $ 0 $, and we now focus on the second term,
\begin{align}
	&\nonumber\int_{a(c)}^{b(c)} f_{L}(x) d \tan ^{-1}\left(\frac{\sqrt{4 c-(x-1-c)^{2}}}{-(x+1-c)}\right)\\
	&=\int_{a(c)}^{b(c)} \frac{\left(x-\log x-1 \right)}{1+\frac{4c-(x-1-c)^{2}}{(x+1-c)^{2}}}\cdot\frac{\sqrt{4c-(x-1-c)^{2}}+\frac{(x-1-c)(x+1-c)}{\sqrt{4c-(x-1-c)^{2}}}}{(x+1-c)^{2}}dx.\label{27}
\end{align}
By substituting $ x=1+c-2\sqrt{c}\cos(\theta) $, we obtain
\begin{align}
	(\ref{27})&=\nonumber\frac{1}{2}\int_{0}^{2\pi}\left(1+c-2\sqrt{c}\cos(\theta)-\log\left(1+c-2\sqrt{c}\cos(\theta) \right)-1  \right)\frac{c-\sqrt{c}\cos(\theta) }{1+c-2\sqrt{c}\cos(\theta)} d\theta\\
	&=\nonumber\frac{1}{2}\int_{0}^{2\pi}\left[1-\frac{\log\left(1+c-2\sqrt{c}\cos(\theta) \right)+1}{1+c-2\sqrt{c}\cos(\theta)}  \right] \left(c-\sqrt{c}\cos(\theta) \right)d\theta\\
	&=\frac{1}{2}\int_{0}^{2\pi}\left(c-\sqrt{c}\cos(\theta) \right)d\theta-\frac{1}{2}\int_{0}^{2\pi}\frac{\log\left(1+c-2\sqrt{c}\cos(\theta) \right)}{1+c-2\sqrt{c}\cos(\theta)}\left( c-\sqrt{c}\cos(\theta)\right)  d\theta-\label{28}\\
	&\nonumber\frac{1}{2}\int_{0}^{2\pi}\frac{c-\sqrt{c}\cos(\theta)}{1+c-2\sqrt{c}\cos(\theta)}d\theta
\end{align}
It is easy to obtain that the first term of (\ref{28}) is $ \pi c $; then, we consider the second term. By substituting $ \cos\theta=\frac{z+z^{-1}}{2} $, we turn it into a contour integral on $ \left|z \right|=1  $
\begin{align*}
	&\frac{1}{2}\int_{0}^{2\pi}\frac{\log\left(1+c-2\sqrt{c}\cos(\theta) \right)}{1+c-2\sqrt{c}\cos(\theta)}\left( c-\sqrt{c}\cos\theta\right)  d\theta\\
	&=\frac{1}{2} \oint_{\left| z\right| =1} \log |1-\sqrt{c} z|^{2} \cdot \frac{c-\sqrt{c} \frac{z+z^{-1}}{2}}{1+c-2\sqrt{ c} \cdot \frac{z+z^{-1}}{2}} \frac{d z}{i z}\\
	&=\frac{1}{4 i} \oint_{\left| z\right| =1} \log |1-\sqrt{c} z|^{2} \cdot \frac{2c z-\sqrt{c}\left(z^{2}+1\right)}{(z-\sqrt{c})(-\sqrt{c}z+1)z} d z
\end{align*}
 When $ c<1 $, $ 0 $ and $\sqrt{c}$ are poles, by using the residue theorem, the integral is $ -\pi \log(1-c)$. The same argument also holds for the third term, and the integral is $ 0 $ after some calculation.

Therefore, $$ \frac{M}{2\pi i}\oint_{\mathcal C}f_{L}\left(z \right)\frac{\underline{m}'(z)}{\underline{m}(z)}dz= -M(c+\log(1-c)), $$ 
and the result is still valid if $ c $ is replaces $ c_{n} $; therefore, formula (\ref{31}) holds.

Now, we prove (\ref{l_{n}}). 
Since $ z=-\frac{1}{\underline{m}}+\frac{c}{1+\underline{m}} $, we have, for $ c>1 $,
\begin{align*}
	\oint_{\mathcal C}f_{W}\left(z \right)\frac{\underline{m}'(z)}{\underline{m}(z)}dz=\oint_{\mathcal C_{1}}f_{W}\left(z \right)\frac{\underline{m}'(z)}{\underline{m}(z)}dz+\oint_{\mathcal C_{2}}f_{W}\left(z \right)\frac{\underline{m}'(z)}{\underline{m}(z)}dz,
\end{align*}
where $\mathcal C_{1}  $ is a contour that includes the interval $ ( (1-\sqrt{c})^{2}, (1+\sqrt{c})^{2}) $, and $\mathcal C_{2}  $ is a contour that includes the origin. Using $ \mathcal C_{\underline{m}} $ to denote the contour of $ \underline{m} $, we obtain
\begin{align*}
	&\oint_{\mathcal C_{1}}f_{W}\left(z \right)\frac{\underline{m}'(z)}{\underline{m}(z)}dz=\oint_{\mathcal C_{\underline{m}}}(-\frac{1}{\underline{m}}+\frac{c}{1+\underline{m}}-1)^{2}\frac{\underline{m}'(z)}{\underline{m}(z)}\frac{dz}{d\underline{m}}d\underline{m}\\
	=&\oint_{\mathcal C_{\underline{m}}}(-\frac{1+\underline{m}}{\underline{m}}+\frac{c}{1+\underline{m}})^{2}\frac{1}{\underline{m}}d\underline{m}=\oint_{\mathcal C_{\underline{m}}}( \frac{(1+\underline{m})^{2}}{\underline{m}^{3}}+\frac{c^{2}}{(1+\underline{m})^{2}\underline{m}}-\frac{2c}{\underline{m}^{2}}  ) d\underline{m}
\end{align*}	 
Since the $ z $ contour cannot enclose the origin, neither can the resulting $ \underline{m} $ contour. Thus, the only pole is $ -1 $, the residue is $ -c^{2} $ by residue theorem, and we obtain the integral as $ -2\pi ic^{2}. $ 

Then, we focus on the second integral $ \oint_{\mathcal C_{2}}f_{W}\left(z \right)\frac{\underline{m}'(z)}{\underline{m}(z)}dz. $ When $ z=0, $ we obtain $ \underline{m}=\frac{1}{c-1} $; since $ c>1,$ $ \frac{1}{c-1}>0. $ Both $ \underline{m}=0 $ and $ \underline{m}=-1 $ are not in the contour. Thus, the integrand $ ( \frac{(1+\underline{m})^{2}}{\underline{m}^{3}}+\frac{c^{2}}{(1+\underline{m})^{2}\underline{m}}-\frac{2c}{\underline{m}^{2}}  ) $ is analytic in the contour. The integral is $ 0 $. 
Therefore, when $ c>1 $,  $ \frac{M}{2\pi i}\oint_{\mathcal C}f_{W}\left(z \right)\frac{\underline{m}'(z)}{\underline{m}(z)}dz=-M c^{2}$. When $ c<1, $, the contour integral $ \oint_{\mathcal C}f_{W}\left(z \right)\frac{\underline{m}'(z)}{\underline{m}(z)}dz $ reduces to $ \oint_{\mathcal C_{1}}f_{W}\left(z \right)\frac{\underline{m}'(z)}{\underline{m}(z)}dz $, and the result is also the same as above. When $ c=1 $, the result is still true by continuity in $ c $. The results above are still valid if $ c $  replaces $ c_{n}. $ Therefore, the proof of (\ref{l_{n}}) is complete.

We now detail the calculations of (\ref{I1}), (\ref{I2}), (\ref{J1}), and (\ref{J2}). They are all based on the formula provided in Remark \ref{simple result} and repeated use of residue theorem.\\
\textit{Proof of (\ref{I1})}:
\begin{align*}
	I_{1}(f_{W})=&\lim _{r \downarrow 1} \frac{1}{2 \pi i} \oint_{|z|=1} \left(\left|1+\sqrt{c} z\right|^{2}-1\right)^{2}\left[\frac{z}{z^{2}-r^{-2}}-\frac{1}{z}\right] d z,\\
	=&\lim _{r \downarrow 1} \frac{1}{2 \pi i} \oint_{|z|=1} \frac{(\sqrt{c}+cz+\sqrt{c}z^{2})^{2}}{z^{2}}\left[\frac{z}{z^{2}-r^{-2}}-\frac{1}{z}\right] d z,\\
	=&\lim _{r \downarrow 1} \frac{1}{2 \pi i} \oint_{|z|=1} \frac{(\sqrt{c}+cz+\sqrt{c}z^{2})^{2}}{z}\frac{1}{(z+\frac{1}{r})(z-\frac{1}{r})} d z\\
	&- \lim _{r \downarrow 1} \frac{1}{2 \pi i} \oint_{|z|=1} \frac{(\sqrt{c}+cz+\sqrt{c}z^{2})^{2}}{z^{3}}d z.
\end{align*}
In the first integral, the poles are $ 0 $, $ -\frac{1}{r} $ and $ \frac{1}{r} $. The residues are $ -r^{2}c, \dfrac{(\sqrt{c}-\frac{c}{r}+\frac{\sqrt{c}}{r^{2}})^{2}}{\frac{2}{r^{2}}}, and \dfrac{(\sqrt{c}+\frac{c}{r}+\frac{\sqrt{c}}{r^{2}})^{2}}{\frac{2}{r^{2}}}. $ In the second integral, the pole is $ 0 $, and the residue is $ c^{2}+2c. $ Then, by using residue theorem, the first part of $ I_{1}(f_{W}) $ is $ 3c+c^{2} $, and the second part is $ c^{2}+2c $; thus, $ I_{1}(f_{W})=c $.\\
\textit{Proof of (\ref{I2})}:
\begin{align*}
	I_{2}(f_{W})&=\frac{1}{2 \pi i} \oint_{|z|=1} \left(\left|1+\sqrt{c} z\right|^{2}-1\right)^{2} \frac{1}{z^{3}} d z,\\
	&=\frac{1}{2 \pi i} \oint_{|z|=1}  \frac{(\sqrt{c}+cz+\sqrt{c}z^{2})^{2}}{z^{5}} d z,
\end{align*}	
The pole is $ 0 $, and the residue is $ c $; then, by using residue theorem, we obtain $ I_{2}(f_{W})=c $.\\
\textit{Proof of (\ref{J1})}:
\begin{align*}
	J_{1}(f_{W}, f_{W})&=\lim _{r \downarrow 1} \frac{-1}{4 \pi^{2}} \oint_{\left|z_{1}\right|=1} \oint_{\left|z_{2}\right|=1} \frac{\left(\left|1+\sqrt{c} z_{1}\right|^{2}-1\right)^{2} \left(\left|1+\sqrt{c} z_{2}\right|^{2}-1\right)^{2}}{\left(z_{1}-r z_{2}\right)^{2}} d z_{1} d z_{2},\\
	&=\lim _{r \downarrow 1} \frac{-1}{4 \pi^{2}}  \oint_{\left|z_{2}\right|=1}\frac{(\sqrt{c}+cz_{2}+\sqrt{c}z_{2}^{2} )^{2}}{z_{2}^{2}}\oint_{\left|z_{1}\right|=1}\frac{(\sqrt{c}+cz_{1}+\sqrt{c}z_{1}^{2} )^{2}}{z_{1}^{2}(z_{1}-rz_{2})^{2}   }dz_{1}dz_{2}.
\end{align*}	
First, we focus on $ \oint_{\left|z_{1}\right|=1}\frac{(\sqrt{c}+cz_{1}+\sqrt{c}z_{1}^{2} )^{2}}{z_{1}^{2}(z_{1}-rz_{2})^{2}   }dz_{1} $. Since $ \left|z_{1} \right|=\left|rz_{2} \right|= \left|r \right|>1, $, the pole is only $ 0 $. By using the residue theorem, the integral is $2\pi i( \frac{2c^{\frac{3}{2}}   }{r^{2} }\frac{ (\sqrt{c}+cz_{2}+\sqrt{c}z_{2}^{2})^{2}}{z_{2}^{4}}+ \frac{2c}{r^{3} }\frac{ (\sqrt{c}+cz_{2}+\sqrt{c}z_{2}^{2})^{2}}{z_{2}^{5}})$. Therefore,
\begin{align*}
	J_{1}(f_{W}, f_{W})&=\lim _{r \downarrow 1} \frac{-1}{4 \pi^{2}} \oint_{\left|z_{2}\right|=1}2\pi i \frac{2c^{\frac{3}{2}}   }{r^{2} }\frac{ (\sqrt{c}+cz_{2}+\sqrt{c}z_{2}^{2})^{2}}{z_{2}^{4}}dz_{2}+\\
	&\lim _{r \downarrow 1} \frac{-1}{4 \pi^{2}} \oint_{\left|z_{1}\right|=1}2\pi i \frac{2c}{r^{3} }\frac{ (\sqrt{c}+cz_{2}+\sqrt{c}z_{2}^{2})^{2}}{z_{2}^{5}}dz_{2}.
\end{align*}	 
For the first integral above, the pole is $ 0 $, and the residue is $ 2c^{\frac{3}{2}}. $ Then, by using residue theorem, the integral is $ 2c^{\frac{3}{2}} $; thus, the first part of $ J_{1}(f_{W}, f_{W}) $ is $ 4c^{3} $. For the second integral, the pole is also $ 0 $, and the residue is $ c $. Similarly, the second integral is $ 2c^{2} $. Therefore,  $J_{1}(f_{W}, f_{W})= 4c^{3}+ 2c^{2}$.\\
\textit{Proof of (\ref{J2})}:
\begin{align*}
	J_{2}(f_{W}, f_{W})&=-\frac{1}{4 \pi^{2}} \oint_{\left|z_{1}\right|=1} \frac{(\left(\left|1+\sqrt{c} z_{1}\right|^{2}\right)-1)^{2}}{z_{1}^{2}} d z_{1} \oint_{\left|z_{2}\right|=1} \frac{(\left(\left|1+\sqrt{c} z_{2}\right|^{2}\right)-1)^{2}}{z_{2}^{2}} d z_{2}.
\end{align*}	
First, we calculate the first integral $ \oint_{\left|z_{1}\right|=1} \frac{(\left(\left|1+\sqrt{c} z_{1}\right|^{2}\right)-1)^{2}}{z_{1}^{2}} d z_{1} $. Since
\begin{align*}
	\oint_{\left|z_{1}\right|=1} \frac{(\left(\left|1+\sqrt{c} z_{1}\right|^{2}\right)-1)^{2}}{z_{1}^{2}} d z_{1}
	=\oint_{\left|z_{1}\right|=1} \frac{ (\sqrt{c}+cz_{1}+\sqrt{c}z_{1}^{2})^{2} }{z_{1}^{4}} d z_{1},
\end{align*}	
the pole is $ 0 $, and the residue is $ 2c^{\frac{3}{2}} $. Then, by using the residue theorem, the integral is $ 4\pi ic^{\frac{3}{2}} $. The same calculations also hold for the second integral and $ \oint_{\left|z_{2}\right|=1} \frac{(\left(\left|1+\sqrt{c} z_{2}\right|^{2}\right)-1)^{2}}{z_{2}^{2}} d z_{2}=4\pi ic^{\frac{3}{2}} $. Therefore, $ J_{2}(f_{W}, f_{W})=4c^{3} $. The proof is finished.


\end{appendix}

\end{document}